\documentclass[11pt, reqno,final]{jnsa}
\usepackage{amsmath, amsthm, amscd, amsfonts, amssymb, graphicx, color}
\usepackage[bookmarksnumbered, plainpages]{hyperref}

\DeclareMathOperator{\dif}{d}
\usepackage{graphicx}
\usepackage{amsthm}
\usepackage{float}
\usepackage{caption}
\usepackage{subcaption}

\begin{document}

\title{ Construction of a global solution for the one dimensional singularly-perturbed boundary value problem}

\author[a1]{Samir Karasulji\' c \corref{c1}}  %
\ead{samir.karasuljic@untz.ba}

\author[a1]{Enes Duvnjakovi\' c}
\ead{enes.duvnjakovic@untz.ba}

\author[a1]{Vedad Pasic}
\ead{vedad.pasic@untz.ba}

\author[a1]{Elvis Barakovic}
\ead{elvis.barakovic@untz.ba}

\address[a1]{Department of Mathematics, Faculty of Natural Sciences and Mathematics, University of Tuzla, Univerzitetska 4, Tuzla, Bosnia and Herzegovina \vskip 0.1cm }

\cortext[c1]{Corresponding author}
\authors{S. Karasulji\' c, E. Duvnjakovi\' c, V. Pasic, E. Barakovic}

\begin{abstract}
We consider an approximate solution for the one--dimensional semilinear singularly--perturbed boundary value problem, using the previously obtained numerical
values of the boundary value problem in the mesh points and the representation of the exact solution using Green's function.
We present an $\varepsilon$--uniform convergence of such gained the approximate solutions, in the maximum norm of the order $\mathcal{O}\left(N^{-1}\right)$ on the observed domain.

After that, the constructed approximate solution is repaired and we obtain a solution, which also has $\varepsilon$--uniform convergence, but now of order $\mathcal{O}\left(\ln^2N/N^2\right)$   on $[0,1].$ In the end a numerical experiment is presented to confirm previously shown theoretical results.


\end{abstract}

\maketitle
\theoremstyle{definition}
\newtheorem{theorem}{Theorem}[section]
\newtheorem{lemma}[theorem]{Lemma}
\newtheorem{proposition}[theorem]{Proposition}
\newtheorem{corollary}[theorem]{Corollary}
\newtheorem{definition}[theorem]{Definition}
\newtheorem{example}[theorem]{Example}
\newtheorem{xca}[theorem]{Exercise}
\newtheorem{problem}[theorem]{Problem}
\theoremstyle{remark}
\newtheorem{remark}[theorem]{Remark}
\numberwithin{equation}{section}


\section{Introduction}

\noindent We will consider the singularly--perturbed boundary value problem
\begin{align}
  \varepsilon^2y''&=f(x,y),\:x\in I=[0,1]\label{uvod1},\\
 & y(0)=0,\:y(1)=0,\label{uvod2}
\end{align}
with the condition
\begin{align}
\dfrac{\partial f(x,y)}{\partial y}:=f_y\geqslant m>0,\:\forall(x,y)\in I\times\mathbb{R},\label{uvod3}
\end{align}
where $0<\varepsilon<1$ is a perturbation parameter, $f$ is a nonlinear function $f\in C^{k}\left(I\times \mathbb{R}\right),\:k\geqslant 2$ and $m$ is a real constant.

\noindent The boundary value problem \eqref{uvod1}--\eqref{uvod2}, with the condition \eqref{uvod3}, has a unique solution, see \cite{lorenz1982stability}. Contributions to  numerical solutions of the problem \eqref{uvod1}--\eqref{uvod2} with different assumptions on the function $f$ and similar problems were obtained by many authors, see for example Flaherty and O'Malley \cite{malley1977numerical}, Cvetkovi{\' c} and Herceg \cite{herceg1982numerical}, Herceg \cite{herceg1982some, herceg1990}, Herceg, Surla and Rapaji{\' c} \cite{herceg1998}, Kopteva \cite{kopteva2001maximum}, Lin{\ss} and Vulanovi{\' c} \cite{vulanovic2001uniform}, Niijima \cite{niijima}, Stynes and O'Riordan  \cite{stynes1987}, Vulanovi{\' c} \cite{vulanovic1983, vulanovic1989, vulanovic1993, vulanovic2004} etc.

\noindent The method that will be used in this paper in order to obtain a discrete approximate solution, i.e. values of the approximate solution in the mesh points, of the problem \eqref{uvod1}--\eqref{uvod3} was first developed by Boglaev \cite{boglaev1984approximate}, who constructed a difference scheme and showed convergence of order 1 on the modified Bakhvalov mesh. Using the method of \cite{boglaev1984approximate}, we constructed new difference schemes in \cite{ samir2011scheme, samir2010scheme} and we carried out numerical experiments.

In \cite{samir2015uniformlyconvergent, samir2015uniformly} we constructed new difference schemes and we proved uniqueness of the numerical solution and $\varepsilon$--uniform convergence on the modified Shishkin mesh and at the end presented numerical experiments. In this paper we will use the difference scheme from \cite{samir2015uniformly} in order to calculate values of the approximate solution of the problem on the mesh points and then construct an approximate solution.

\section{Theoretical background and known results}

\noindent Let us set up an arbitrary mesh on $[0,1]$
\begin{equation}
   0=x_0<x_1<\ldots<x_N=1.
 \label{thb1}
\end{equation}

\noindent A construction of a difference scheme, which will be used for calculation of the approximate solution of the problem \eqref{uvod1}--\eqref{uvod3} in the mesh points, is based on the representation of the exact solution on the interval $[x_i,x_{i+1}],\:i=0,\ldots,N-1$
\begin{equation}
 y_i(x)=y_iu_i^{I}(x)+y_{i+1}u_i^{II}(x)+\int_{x_i}^{x_{i+1}}{G_i(x,s)\psi(s,y(s))\dif s},
\label{thb2}
\end{equation}
where $G_i(x,s)$ is the Green's function
\begin{equation}
  G_i(x,s)=\frac{1}{\varepsilon^2w_i(s)}
        \left\{\begin{array}{cc}
          u_{i}^{II}(x)u_i^{I}(s),\:&x_i\leqslant x\leqslant s\leqslant x_{i+1},\\\\
          u_i^{I}(x)u_i^{II}(s), \:&x_i\leqslant s\leqslant x\leqslant x_{i+1},
        \end{array}
       \right.
\label{thb3}
\end{equation}
\begin{equation}
 \psi(s,y(s))=f(s,y(s))-\gamma y(s),
  \label{thb3a}
\end{equation}
and $w_i(s)=\frac{-\beta}{\sinh(\beta h_i)},$ $s\in[x_{i},x_{i+1}],$ $u_i^{I}(x)=\frac{\sinh(\beta(x_{i+1}-x))}{\sinh(\beta h_i)},$ $u_i^{II}(x)=\frac{\sinh(\beta(x-x_i))}{\sinh(\beta h_i)},$ $h_i=x_{i+1}-x_i,\:\beta=\frac{\sqrt{\gamma}}{\varepsilon},\,y_i:=y(x_i)$ and  $\gamma$ is a constant for which  $\gamma\geqslant f_y,$ (details can be found in \cite{samir2015uniformly}).
The difference scheme constructed in \cite{samir2015uniformly}, which we will use, has the following form
 \begin{equation}
 \dfrac{a_i+d_i}{2}\overline{y}_{i-1}-\left( \dfrac{a_i+d_i}{2}+\dfrac{a_{i+1}+d_{i+1}}{2}\right)\overline{y}_{i}+ \dfrac{a_{i+1}+d_{i+1}}{2}\overline{y}_{i+1}
                        =\dfrac{\triangle d_{i}}{\gamma}\overline{f}_{i-1}+\dfrac{\triangle d_{i+1}}{\gamma}\overline{f}_{i},
\label{thb4}
\end{equation}
where $\overline{y}_k,\:k\in\left\{i-1,i,i+1\right\}$ are values of the approximate solution in the mesh points, $\triangle d_i=d_i-a_i,$ $d_i=\frac{1}{\tanh(\beta h_i)},$ $a_i=\frac{1}{\sinh(\beta h_i)}$ and $\overline{f}_i=f((x_{i}+x_{i+1})/2,(\overline{y}_i+\overline{y}_{i+1})/2),\:i=1,\ldots,N-1.$ The difference scheme generates a system of nonlinear equations and the solutions of this system are values of the approximate solution in the mesh points. An answer to the question of existence and uniqueness will be given in the next theorem, however before that, it is necessary to define the operator (or discrete problem) $F:\mathbb{R}^{N+1}\mapsto\mathbb{R}^{N+1}$ and a corresponding norm that is necessary in formulation of the theorem. Therefore, we will now use the difference scheme \eqref{thb4} in order to obtain a discrete problem of the problem \eqref{uvod1}--\eqref{uvod3}. We have that
\begin{align}
  F\overline{y}=\left( \left(F\overline{y}\right)_0,\left(F\overline{y}\right)_1,\ldots,\left( F\overline{y}\right)_N \right)^T=0,
 \label{thb5}
\end{align}
where
\begin{eqnarray*}
  \left(F\overline{y}\right)_0&:=&0,\nonumber\\
  \left(F\overline{y}\right)_i&:=& \frac{\gamma}{\triangle d_i+\triangle d_{i+1}}
               \left[ \dfrac{a_i+d_i}{2}\overline{y}_{i-1}-\left( \dfrac{a_i+d_i}{2}+\dfrac{a_{i+1}+d_{i+1}}{2}\right)\overline{y}_{i}\right. \nonumber \\
               & &\quad     +\left. \dfrac{a_{i+1}+d_{i+1}}{2}\overline{y}_{i+1}
                  -\dfrac{\triangle d_{i}}{\gamma}\overline{f}_{i-1}-\dfrac{\triangle d_{i+1}}{\gamma}\overline{f}_{i}\right],\:i=1,\ldots,N-1\nonumber\\
  \left(F\overline{y}\right)_N&:=&0. \nonumber
\end{eqnarray*}

\noindent Here we use the maximum norm
 \begin{equation}
      \left\|u\right\|_{\infty}=\max_{0\leqslant i\leqslant N}\left|u_i\right|,
   \label{thb6}
 \end{equation}
for any vector $u=\left(u_0,u_1,\ldots,u_n\right)^T\in\mathbb{R}^{N+1}$ and the corresponding matrix norm.

\begin{theorem}{\rm \cite{samir2015uniformly}} The discrete problem $(\ref{thb5})$ for $\gamma\geq f_y,$ has  the unique solution\\ $\overline{y}=(\overline{y}_0, \overline{y}_1, \overline{y}_2, \ldots, \overline{y}_{N-1}, \overline{y}_{N})^{T},$ with $\overline{y}_0=\overline{y}_N=0.$ Moreover, the following stability inequality holds
\begin{equation}
     \left\|w-v\right\|_{\infty}\leqslant \frac{1}{m}\left\|Fw-Fv\right\|_{\infty},
 \label{thb7}
\end{equation}
for any vectors $v=\left(v_0,v_1,\ldots,v_N\right)^T\in\mathbb{R}^{N+1},\,w=\left( w_0,w_1,\ldots,w_N\right)^T\in\mathbb{R}^{N+1}.$
\end{theorem}
The mesh that will be used here is a modified Shishkin mesh from \cite{linss2010, linss2012approximation}, which has a greater smoothness compared to the generating function. Before the construction of the mesh, we are stating a theorem about the decomposition and estimates of the derivatives, which is necessary for the construction and further analysis.
\begin{theorem}\label{teorema2}{\rm\cite{vulanovic1983numerical}}
The solution $y$ to the problem \eqref{uvod1}--\eqref{uvod3} can be represented in the following way
  \begin{equation}
    y=r+s,
      \label{jed6}
  \end{equation}
where for $i=0,1,\ldots,k$ and $x\in[0,1]$ we have that
  \begin{subequations}
    \begin{equation}
        \left|r^{(i)}(x)\right|\leqslant C,
       \label{jed7a}
    \end{equation}
    \begin{equation}
     \left|s^{(i)}(x)\right|\leqslant C\varepsilon^{-i}\left(e^{-\frac{x}{\varepsilon}\sqrt{m}}+e^{-\frac{1-x}{\varepsilon}\sqrt{m}}\right).
     \label{jed7b}
    \end{equation}
  \end{subequations}
\end{theorem}

\noindent Let  $N+1$ be the number of mesh points, $q\in(0,1/2)$ and $\sigma>0$ be the mesh parameter. We will define the transition point of the Shishkin mesh with
\begin{equation*}
  \lambda:=\min\left\{ \frac{\sigma\varepsilon}{\sqrt{m}}\ln N,q\right\}.
\label{thb8}
\end{equation*}
Let $\sigma=2.$
\begin{remark} For  the sake of simplicity in representation, we assume that $\lambda=2\varepsilon (\sqrt{m})^{-1}\ln N$, as otherwise the problem can be analysed in the classical way. We shall also assume that $q N$ is an integer. This is easily achieved by choosing $q=1/4$ and $N$ divisible by $4$ for example.
\end{remark}

\noindent The mesh $\triangle:x_0<x_1<...<x_N$ is generated by $x_i=\varphi(i/N)$ with the mesh generating function
\begin{equation}
\varphi(t):=\left\{
              \begin{array}{ll}
                 \tfrac{\lambda}{q}t &t\in[0,q],\\
                 p(t-q)^3  +\lambda          &t\in[q,1/2],\\
                 1-\varphi(1-t)         &t\in[1/2,1],
              \end{array}
      \right.
\label{thb9}
\end{equation}
where $p$  is chosen so that $\varphi(1/2)=1/2,$ i.e. $p=\tfrac{1}{2}(1-\tfrac{\lambda}{q})(\tfrac{1}{2}-q)^{-3}.$ Note that $\varphi\in C^{1}[0,1]$ with $\left\|\varphi'\right\|_{\infty},\left\|\varphi''\right\|_{\infty}\leq C.$ Therefore the mesh sizes $h_{i}=x_{i+1}-x_{i},\,i=0,1,...,N-1$ satisfy
\begin{equation}
  h_i\leqslant \frac{C}{N} \text{\quad and \quad}   |h_{i+1}-h_i|\leqslant \frac{C}{N^{2}},
\label{thb10}
\end{equation}
see \cite{linss2012approximation} for details.

\begin{theorem}{\rm \cite{samir2015uniformly}} The difference scheme \eqref{thb4} on the mesh generated by the function \eqref{thb9}  is uniformly convergent with respect to $\varepsilon$ and
\begin{equation*}
\max\limits_{0\leq i\leq N}\left|y(x_i)-\overline{y}_i\right|\leq C\dfrac{\ln^2 N}{N^2},
\label{thb11}
\end{equation*}
where $y(x)$ is the solution of the problem \eqref{uvod1}--\eqref{uvod3}, $\overline{y}$ is the corresponding numerical solution of \eqref{thb5}, and $C>0$ is a constant independent of $N$ and $\varepsilon$.
\end{theorem}

\section{Main results}

\noindent On the interval  $[x_i,x_{i+1}]$ using the representation \eqref{thb2}, we look for an approximate solution in the following form
\begin{equation}
 \tilde{y}_i(x)= \overline{y}_iu_i^{I}(x)+\overline{y}_{i+1}u_i^{II}(x)+\overline{\psi}_i\int_{x_i}^{x_{i+1}}{G_i(x,s)\dif s},\:i=0,\ldots,N-1,
\label{mr2}
\end{equation}
where
\begin{equation}
  \overline{\psi}_i=\psi((x_i+x_{i+1})/2,(\overline{y}_i+\overline{y}_{i+1})/2),\:i=0,1,\ldots,N-1.
 \label{mr1}
\end{equation}
We obtain that it is
\begin{multline}
 \int_{x_i}^{x_{i+1}}{G_i(x,s)\dif s}=-\frac{\sinh(\beta(x_{i+1}-x))}{\gamma\sinh(\beta h_i)}\left[ \cosh(\beta(x-x_i))-1\right]\\
                                                   - \frac{\sinh(\beta(x-x_i))}{\gamma\sinh(\beta h_i)}\left[ \cosh(\beta(x_{i+1}-x))-1\right],\:i=0,\ldots,N-1.
\label{mr3}
\end{multline}

\noindent We are looking for an approximate solution on $[0,1]$ in the form
\begin{equation}
 \Bigl.Y(x)\Bigr|_{[x_i,x_{i+1}]}=\tilde{y}_i(x),\:i=0,\ldots,N-1.
 \label{aprox}
\end{equation}

\noindent Using the maximum norm, we estimate the difference between the exact solution of the problem \eqref{uvod1}--\eqref{uvod3} and approximate solutions given by \eqref{aprox}.
This difference will be estimated on each interval $[x_i,x_{i+1}],\:i=0,\ldots,N-1.$ Taking into account \eqref{thb2}, \eqref{mr2} and  \eqref{aprox}, we have that
\begin{multline}
   \left|y_i(x)-\tilde{y}_i(x)\right|\leqslant\left|y_i-\overline{y}_i\right|\left|u^{I}_{i}(x)\right|+\left|y_{i+1}-\overline{y}_{i+1}\right|\left|u^{II}_i(x)\right|\\
    +\left|\int_{x_i}^{x_{i+1}}{  G_i(x,s)\left( \psi(s,y(s))-\overline{\psi}_i\right) \dif s}\right|,\:i=0,\ldots,N-1.
 \label{mr5}
\end{multline}
\begin{remark}\label{remark1}
An estimate of the value of difference $\left|y(x)-Y(x)\right|,\:\forall x\in[0,1],$ or estimate of the error will be done for $[0,1/2].$ An analogue estimate would hold on $[1/2,1].$

\noindent Note that $e^{-x\sqrt{m}/\varepsilon}\geqslant e^{-(1-x)\sqrt{m}/\varepsilon}$ and $h_{i+1}\geqslant h_i$ for $x\in[0,1/2]$ and $e^{-x\sqrt{m}/\varepsilon}\leqslant e^{-(1-x)\sqrt{m}/\varepsilon}$ and $h_{i+1}\leqslant h_i$ for $x\in[1/2,1].$
\end{remark}

\noindent Let us first estimate $\displaystyle\int_{x_i}^{x_{i+1}}{G_i(x,s)\dif s}$ for $x\in[0,\lambda].$

\begin{lemma}\label{lemaeks1}
For  $x\in[x_i,x_{i+1}],\:i=0,\ldots,N/4-1,$ we have the following estimate
\begin{equation}
\left| \frac{\sinh(\beta(x_{i+1}-x))}{\gamma\sinh(\beta h_i)}\left[ \cosh(\beta(x-x_i))-1\right]\right.
           \left. + \frac{\sinh(\beta(x-x_i))}{\gamma\sinh(\beta h_i)}\left[ \cosh(\beta(x_{i+1}-x))-1\right]\right|\leqslant\frac{C\ln^2N}{N^2}.
 \label{lemaeks2}
\end{equation}
\end{lemma}
\begin{proof}
\begin{align*}
&\frac{\sinh(\beta(x_{i+1}-x))}{\gamma\sinh(\beta h_i)}\left[ \cosh(\beta(x-x_i))-1\right]
+ \frac{\sinh(\beta(x-x_i))}{\gamma\sinh(\beta h_i)}\left[ \cosh(\beta(x_{i+1}-x))-1\right]    \\
&=\frac{\sinh(\beta (x_{i+1}-x_i))-\sinh(\beta(x_{i+1}-x))-\sinh(\beta(x-x_i))}{\gamma\sinh(\beta h_i)}                                            \\
&=\frac{\beta h_i+\frac{\beta^3h_i^3}{6}+\mathcal{O}_1\left(\beta^5h^5_i \right)
           -\beta(x_{i+1}-x)-\frac{\beta^3(x_{i+1}-x)^3}{6}-\mathcal{O}_2\left(\beta^5(x_{i+1}-x)^5 \right)}
       {\gamma\left[ \beta h_i+\frac{\beta^3h^3_i}{6}+\mathcal{O}_1\left(\beta^5h^5_i\right)\right]}                  \\
&\hspace{.5cm} -\frac{\beta(x-x_i)+\frac{\beta^3(x-x_i)^3}{6}+\mathcal{O}_3\left(\beta^5(x-x_i)^5 \right)}
         {\gamma\left[ \beta h_i+\frac{\beta^3h^3_i}{6}+\mathcal{O}_ 1\left(\beta^5h^5_i\right)\right]}
\end{align*}
\begin{align*}
&=\frac{\frac{1}{2}\beta^3(x-x_i)(x-x_{i+1})(x_i-x_{i+1})}
              {\gamma\left[ \beta h_i+\frac{\beta^3h^3_i}{6}+\mathcal{O}_1\left(\beta^5h^5_i\right)\right]}  \\
&\hspace{3.35cm}  +\frac{\mathcal{O}_1\left(\beta^5h^5_i \right)-\mathcal{O}_2\left(\beta^5(x_{i+1}-x)^5 \right)-\mathcal{O}_3\left(\beta^5(x-x_i)^5 \right)}
              {\gamma\left[ \beta h_i+\frac{\beta^3h^3_i}{6}+\mathcal{O}_1\left(\beta^5h^5_i\right)\right]}   .
\end{align*}
Furthermore, based on the value of parameter $\beta$ and the properties of the mesh, we have that
\begin{multline}
 \left|\frac{\frac{1}{2}\beta^3(x-x_i)(x-x_{i+1})(x_i-x_{i+1})}
              {\gamma\left[ \beta h_i+\frac{\beta^3h^3_i}{6}+\mathcal{O}_1\left(\beta^5h^5_i\right)\right]} \right.\\
      +\left.\frac{\mathcal{O}_1\left(\beta^5h^5_i \right)-\mathcal{O}_2\left(\beta^5(x_{i+1}-x)^5 \right)-\mathcal{O}_3\left(\beta^5(x-x_i)^5 \right)}
              {\gamma\left[ \beta h_i+\frac{\beta^3h^3_i}{6}+\mathcal{O}_1\left(\beta^5h^5_i\right)\right]} \right| \\
         \leqslant C_1\frac{\frac{\ln^3N}{N^3}+\frac{\ln^5N}{N^5}}{\frac{\ln N}{N}}\leqslant\frac{C\ln^2 N}{N^2} .
 \label{eksp6}
\end{multline}
Now, using \eqref{eksp6}, we obtain \eqref{lemaeks2}.
\end{proof}

\begin{lemma}\label{lemaeksp3}
For $x\in[x_i,x_{i+1}],\:i=N/4,\ldots,N/2-1,$ we have the following estimate
\begin{equation}
 \left| \frac{\sinh(\beta(x_{i+1}-x))}{\gamma\sinh(\beta h_i)}\left[ \cosh(\beta(x-x_i))-1\right]\right.
           \left. + \frac{\sinh(\beta(x-x_i))}{\gamma\sinh(\beta h_i)}\left[ \cosh(\beta(x_{i+1}-x))-1\right]\right|\leqslant C.
 \label{ekps7}
\end{equation}
\end{lemma}
\begin{proof}
In the proof of the Lemma \ref{lemaeks1}, it is shown that
\begin{multline}
\frac{\sinh(\beta(x_{i+1}-x))}{\gamma\sinh(\beta h_i)}\left[ \cosh(\beta(x-x_i))-1\right]\\
                                                   + \frac{\sinh(\beta(x-x_i))}{\gamma\sinh(\beta h_i)}\left[ \cosh(\beta(x_{i+1}-x))-1\right]\\
=  \frac{\sinh(\beta (x_{i+1}-x_i))-\sinh(\beta(x_{i+1}-x))-\sinh(\beta(x-x_i))}{\gamma\sinh(\beta h_i)} .
\end{multline}
We get that
\begin{multline}
  \left| \frac{\sinh(\beta (x_{i+1}-x_i))-\sinh(\beta(x_{i+1}-x))-\sinh(\beta(x-x_i))}{\gamma\sinh(\beta h_i)}\right|\\
  \leqslant\frac{1}{\gamma}\left(1+\left|\frac{\sinh(\beta(x_{i+1}-x))}{\sinh(\beta h_i)}\right|+\left| \frac{\sinh(\beta(x-x_i))}{\sinh(\beta h_i)}\right| \right)
  \leqslant C.
 \label{eksp8}
\end{multline}
\end{proof}

\begin{theorem}\label{th1}
Let  $y$ be the exact solution of the problem \eqref{uvod1}--\eqref{uvod3},  and $Y$ be the appropriate approximate solution given in \eqref{aprox}. We have the following estimate
\begin{equation}
   \max_x\left|y(x)-Y(x)\right|\leqslant C
      \left\{ \begin{array}{cl}
          \dfrac{\ln^2N}{N^2},&\:x\in[0,\lambda],\\\\
          \dfrac{1}{N},&\:x	\in[\lambda,1-\lambda],\\\\
          \dfrac{\ln^2N}{N^2},&\:x\in [1-\lambda,1],
       \end{array}\right.
 \label{th2}
\end{equation}
where the constant $C$ does not depend on the perturbation parameter $\varepsilon$ nor $N.$
\end{theorem}
\begin{proof}
We divide $[0,1]$ by the mesh points $x_i,\,i=1,\ldots,N-1$ into subintervals $[x_i,x_{i+1}],\,i=0,\ldots,N$. Since $Y(x)=\tilde{y}_i(x)$ on $[x_i,x_{i+1}],$
we estimate the difference  $\left|y(x)-\tilde{y}_i(x)\right|$ on each subinterval $[x_i,x_{i+1}]$. Based on representations of the exact solution \eqref{thb2} and the approximate solution \eqref{mr2} on the interval $[x_i,x_{i+1}],$ we have that the estimate \eqref{mr5} holds and
\begin{equation}
  \left| y_i(x)-\tilde{y}_i(x)\right|
            \leqslant  \left|y_i-\overline{y}_i\right|u^{I}_i(x)+\left|y_{i+1}-\overline{y}_{i+1}\right|u^{II}_i(x) +\left|\int_{x_i}^{x_{i+1}}{G_i(x,s)\left[\psi(s,y(s))- \overline{\psi}_i\right]\dif s} \right|.
 \label{th3}
\end{equation}


\noindent Let us first estimate the difference  $\psi(x,y(x))- \overline{\psi}_i$ on the interval $[x_i,x_{i+1}],\:i=0,\ldots,N/4-1$, which appears in the integrand in \eqref{th3}. Using Lagrange's theorem we obtain
\begin{align}
  \left|\psi(x,y(x))- \overline{\psi}_i\right|=&\left|f(x,y(x))-f\left(\tfrac{x_i+x_{i+1}}{2},\tfrac{\overline{y}_i+\overline{y}_{i+1}}{2}\right)
              -\gamma\left(y(x)-\frac{\overline{y}_i+\overline{y}_{i+1}}{2} \right)\right|  \nonumber\\
    =&\left|\left(\frac{\partial f(\xi,\eta)}{\partial y}-\gamma\right)\left(y(x)-\frac{\overline{y}_i+\overline{y}_{i+1}}{2}\right)
       +\frac{\partial f(\xi,\eta)}{\partial x}\left(x-\frac{x_i+x_{i+1}}{2}\right)\right|\nonumber\\
    \leqslant&\frac{C\ln N}{N}.
 \label{th4}
\end{align}
Let now $i=N/4+1,\ldots,N/2-1.$  We have that
\begin{align}
  \left|\psi(x,y(x))- \overline{\psi}_i\right|=&\left|f(x,y(x))-f\left(\tfrac{x_i+x_{i+1}}{2},\tfrac{\overline{y}_i+\overline{y}_{i+1}}{2}\right)
         -\gamma\left(y(x)-\frac{\overline{y}_i+\overline{y}_{i+1}}{2} \right)  \right|\nonumber\\
     =&\left|\left(\frac{\partial f(\xi,\eta)}{\partial y}-\gamma\right)\left(y(x)-\frac{\overline{y}_i+\overline{y}_{i+1}}{2}\right)
       +\frac{\partial f(\xi,\eta)}{\partial x}\left(x-\frac{x_i+x_{i+1}}{2}\right)\right|\nonumber\\
    \leqslant&\frac{C}{N},
 \label{th5}
\end{align}
where $\xi\in\left(x,(x_i+x_{i+1})/2\right)$ or $\xi\in((x_i+x_{i+1})/2,x)$ in \eqref{th4}, and  $\eta\in (y,(y_i+y_{i+1})/2)$
or  $\eta\in((y_i+y_{i+1})/2,y)$ in \eqref{th5}.\\\\

\noindent Let us estimate another difference $\psi(x,y(x))- \overline{\psi}_i$ on the interval $\left[ N/4,N/4+1\right].$
Since $\varepsilon^2y''(x)=f(x,y(x))$, we get the estimate
\begin{align}
  \left|f(x,y(x))-f\left(\tfrac{x_i+x_{i+1}}{2},\tfrac{\overline{y}_i+\overline{y}_{i+1}}{2} \right)\right|
            &\leqslant\left|f(x,y(x)\right|+\left|f\left(\tfrac{x_i+x_{i+1}}{2},\tfrac{\overline{y}_i+\overline{y}_{i+1}}{2} \right)\right|
              \leqslant\frac{C}{N^2}.
   \label{th6}
\end{align}
Now, from $\left|y(x_i)-\overline{y}_i\right|\leqslant\frac{C\ln^2 N}{N^2},\:i=0,\ldots,N,$  and decomposition and estimates from Theorem \ref{teorema2}, we get the following estimate
\begin{align}
\left|y(x)-\frac{\overline{y}_i+\overline{y}_{i+1}}{2}\right|\leqslant&\left|y(x)-\frac{y(x_i)+y(x_{i+1})}{2}\right|+\frac{C_1\ln^2 N}{N^2}\nonumber\\
                  \leqslant&\left| s(x)-\frac{s(x_i)+s(x_{i+1})}{2}\right| +\left|r(x)-\frac{r(x_i)+r(x_{i+1})}{2}\right|+\frac{C_1\ln^2 N}{N^2}\nonumber\\
                  \leqslant&\frac{C_2}{N^2}+\left|r'(\mu)\right|\left(x-\frac{x_i+x_{i+1}}{2}\right)+\frac{C_1\ln^2 N}{N^2}\leqslant\frac{C}{N},
 \label{th7}
\end{align}
where $\mu\in(x,(x_i+x_{i+1})/2)$ or $\mu\in((x_i+x_{i+1})/2,x).$
Now from  \eqref{mr3}, Lemma \ref{lemaeks1}, Lemma \ref{lemaeksp3}, and the estimates \eqref{th4}, \eqref{th5}, \eqref{th6} and  \eqref{th7}  the 	
assertion of the theorem follows.

\end{proof}

\noindent According the proof of the previous theorem it is  shown that  the difference between the exact and approximate solution  $\left| y(x)-Y(x)\right|$ on $[0,\lambda]$ is of the order $\mathcal{O}\left(\ln^2 N/N^2\right),$ while on $[\lambda,1-\lambda]$ that order of the error is  $\mathcal{O}\left(1/N\right).$ Based on the Theorem \ref{thb11}, the difference between the exact and the approximate solution on the mesh points is of order  $\mathcal{O}\left(\ln^2N/N^2\right).$ In order to get the approximate  solution with a satisfactory value of the error, we must conduct the correction of the approximate  solutions given in \eqref{mr2}. Namely, since this constructed approximate solution performs well at the layer, which is the most problematic part of the analysis, we will take on this part the approximate solution which was given in \eqref{mr2}. In the remaining part of the observed domain, i.e. for $x\in\left[\lambda,1-\lambda\right]$ we will use a piecewise linear function.\\\\
Therefore, for $x\in[0,\lambda]\cup[1-\lambda,1],$ we use
\begin{equation}
  \tilde{y}_i(x)=\overline{y}_iu^{I}_i(x)+\overline{y}_{i+1}u^{II}_i(x)+\int_{x_i}^{x_{i+1}}{G_i(x,s)\psi(x_i,\overline{y})\dif s},
 \label{eksp9}
\end{equation}
while for $x\in[\lambda,1-\lambda],$  we use the following interpolation polynomial
\begin{equation}
 \overline{ p}(x)=\left\{
               \begin{array}{cl}
                 \overline{ p}_{N/4}(x)\quad &x\in[x_{N/4},x_{N/4+1}],\\
                \vdots\hspace{.3cm}&\\
                 \overline{ p}_{i}(x)\quad& x\in[x_{i},x_{i+1}],\\
                       \vdots\hspace{.3cm}&\\
                  \overline{p}_{3N/4-1}(x)\quad& x\in[x_{3N/4-1},x_{3N/4}] ,
               \end{array}
       \right.
 \label{eksp10}
\end{equation}
where
\begin{equation}
   \overline{ p}_i(x)=\left\{ \begin{aligned}
              \dfrac{ \overline{y}_{i+1}-\overline{y}_i}{x_{i+1}-x_i}&(x-x_i)+\overline{y}_i\quad &x\in[x_i,x_{i+1}],\\\\
              &0 &x \not\in[x_i,x_{i+1}]
               \end{aligned}
               \right.
 \label{eksp11}
\end{equation}
and $\overline{y}_i,\:i=N/4,\ldots,3N/4-1$ are the already calculated values of the approximate solutions in the mesh points. Now, the approximate solution to the problem  \eqref{uvod1}--\eqref{uvod3}, has the following form
\begin{equation}
  \widetilde{Y}(x)=\left\{
         \begin{array}{cl}
            \tilde{y}_i(x) & x\in[0,\lambda],\\\\
            \overline{p}(x)& x\in[\lambda,1-\lambda],\\\\
            \tilde{y}_i(x) & x\in[1-\lambda,1].
         \end{array}
       \right.
\label{eksp12}
\end{equation}

\begin{remark}
In the following theorem, the estimate of the error will be calculated only for  $x\in[\lambda,1/2],$ i.e. for the value of the indexes $i=N/4,\ldots,N/2.$ We use the same assumptions as previously listed in {\rm Remark \ref{remark1}}.
\end{remark}

\begin{theorem}\label{th8}
The following estimate of the error between the exact and approximate solution \eqref{uvod1}--\eqref{uvod3} holds:
\begin{equation}
  \max_{x\in[0,1]} \left|y(x)-\widetilde{Y}(x)\right|\leqslant\frac{C\ln^2N}{N^2}.
 \label{th9}
\end{equation}
\end{theorem}

\begin{proof}
The case of $x\in[0,\lambda]$ has already been proved in the Theorem \ref{th1}. \\

\noindent Let us show now \eqref{th9} on $[\lambda,1/2].$ Let us denote by  $p$ a polynomial which is defined in the same way as the polynomial $\overline{p}$ in \eqref{eksp10}--\eqref{eksp11}. The polynomial $p$ will pass through the points with coordinates  $(x_i,y_i)$ and $(x_{i+1},y_{i+1}),$ ($y_i$ and  $y_{i+1}$ are values of the exact solution in the mesh points, i.e. $y_i=y(x_i),\,y_{i+1}=y(x_{i+1})).$ We have that
\begin{equation}
  \left|y(x)-\overline{p}(x)\right|=\left|y(x)-p(x)+p(x)-\overline{p}(x)\right|\leqslant\left|y(x)-p(x)\right|+\left|p(x)-\overline{p}(x)\right|.
 \label{th10}
\end{equation}

\noindent On every interval $[x_i,x_{i+1}],\:i=N/4,\ldots,N/2,$ we get that
\begin{align}
   p(x)-\overline{p}(x)=
      &\dfrac{ y_{i+1}-y_i}{x_{i+1}-x_i}(x-x_i)+y_i-\dfrac{\overline{y}_{i+1}-\overline{y}_i}{x_{i+1}-x_i}(x-x_i)-\overline{y}_i\nonumber\\
     =&\frac{y_{i+1}-\overline{y}_{i+1}-(y_i-\overline{y}_i)}{x_{i+1}-x_i}(x-x_i)-(y_i-\overline{y}_i),
\end{align}
therefore in view of the Theorem  \ref{thb11} we obtain the estimate
\begin{equation}
   \left|p(x)-\overline{p}(x)\right|\leqslant\frac{C\ln^2N}{N^2},\:i=N/4,\ldots,N/2.
 \label{th11}
\end{equation}

\noindent In the part of the mesh when $i=N/4+1,\ldots,N/2,$  on basis of \cite[Example 8.12]{kress1998numerical}, \eqref{jed7a}, \eqref{jed7b} and \eqref{thb10}, we obtain
 \begin{align}
    \left|y-p_i(x)\right|\leqslant \frac{h^2}{8}\max_{\eta\in[x_i,x_{i+1}]}\left|y''(\eta)\right|
        \leqslant \frac{C}{N^2}.
  \label{jed15}
 \end{align}

\noindent For $i=N/4,$ according to the  decomposition  \eqref{jed6} from Theorem \ref{teorema2},  we obtain
\begin{align}
   y-p_i(x)=&y-\dfrac{ y_{i+1}-y_i}{x_{i+1}-x_i}(x-x_i)+y_i\\
            =&s-\dfrac{ s_{i+1}-s_i}{x_{i+1}-x_i}(x-x_i)+s_i+r-\dfrac{ r_{i+1}-r_i}{x_{i+1}-x_i}(x-x_i)+r_i.
 \label{jed16}
\end{align}
For the layer component, on the basis of the estimate  \eqref{jed7b} we obtain
\begin{multline}
    \left|s-\dfrac{ s_{i+1}-s_i}{x_{i+1}-x_i}(x-x_i)+s_i\right|\leqslant |s|+|s_{i+1}-s_{i}|+|s_i|\\
         \leqslant C_1\left[ e^{-\frac{x}{\varepsilon}\sqrt{m}}+e^{-\frac{1-x}{\varepsilon}\sqrt{m}}
            +\left(e^{-\frac{x_{i+1}}{\varepsilon}\sqrt{m}}+e^{-\frac{1-x_{i+1}}{\varepsilon}\sqrt{m}}\right)\right.\\
            +\left.2\left( e^{-\frac{x_i}{\varepsilon}\sqrt{m}}+e^{-\frac{1-x_i}{\varepsilon}\sqrt{m}}\right) \right]
         \leqslant \frac{C}{N^2}.
 \label{jed17}
\end{multline}
For the regular component we apply again the estimate from  \cite[Example 8.12]{kress1998numerical}, and on the basis of \eqref{jed7a} we get that
\begin{equation}
   \left|r-\dfrac{ r_{i+1}-r_i}{x_{i+1}-x_i}(x-x_i)+r_i\right|\leqslant  \frac{h^2}{8}\max_{\eta\in[x_{i},x_{i+1}]}\left| y''(\eta)\right|\leqslant\frac{C}{N^2}.
 \label{jed18}
\end{equation}
\noindent Now, from \eqref{th11}, \eqref{jed15}, \eqref{jed17} and \eqref{jed18}, and the part of the proof of Theorem \ref{th1}, which is related to $x\in[0,\lambda],$ we obtain \eqref{th9}.
\end{proof}

\section{Numerical Experiments}
\noindent In this section the theoretical results of the previous section will be checked on the following example
\begin{equation}
  \varepsilon^2y''=y+(1-2x)^2-8\varepsilon^2,\:x\in(0,1),\:y(0)=0,\:y(1)=0.
 \label{num1}
\end{equation}
The exact solution of the test example \eqref{num1} is
\begin{equation}
  y(x)=\frac{e^{-x/\varepsilon}+e^{-(1-x)/\varepsilon}}{1+e^{-1/\varepsilon}}+4x(1-x)-1.
\label{num2}
\end{equation}
First we will calculate a discrete approximate solution, i.e. the value of approximate solutions in the mesh points, using the difference scheme \eqref{thb4} and then based on those results we will construct approximate solutions \eqref{mr2} and \eqref{eksp12}. Plots of exact and approximate solutions \eqref{mr2} and  \eqref{eksp12} are represented by Figure \ref{slika56z} and Figure \ref{slika1718z}, while the values of errors are presented in na Figure \ref{slikagreska}.

\noindent The system of equations is solved by Newton's method with initial guess $y_0=-0.5.$ The value of the constant $\gamma=1$ has been chosen so that the condition $\gamma\geq f_{y}(x,y),\:\forall(x, y)\in\left[0,1\right]\times \mathbb{R}$ is satisfied.
Because of the fact that we know the exact solution, we define  the computed error $E_N$ and the computed rate of convergence Ord in the usual way
\begin{equation*}
E_N=\max\limits_{0\leq i\leq N}\left|y(x_i)-\overline{y}^N(x_i) \right|, \qquad
\text{Ord}=\dfrac{\ln E_N-\ln E_{2N}}{\ln\frac{2k}{k+1}},
\label{problem8}
\end{equation*}
where  $N=2^{k},\:k=5,6,\ldots,11,$ $\overline{y}^N(x_i)$ is the numerical solution on a mesh with $N$ subintervals. Values $E_N$ and  Ord are represented in the following table.
\begin{center}
\begin{table}[h]\footnotesize
\centering
\begin{tabular}{c|cc|cc|cc}\hline
     $N$ &$E_n$&Ord&$E_n$&Ord&$E_n$&Ord\\\hline
$2^{5}$&$4.9836e-03$&$2.01$   &$1.8622e-02$&$2.95$    &$1.9923e-02$&$2.55$ \\
$2^{6}$&$1.7834e-03$&$1.98$   &$4.1194e-03$&$2.01$    &$5.4155e-03$&$2.00$ \\
$2^{7}$&$6.1200e-04$&$2.00$   &$1.3925e-03$&$2.00$    &$1.8429e-03$&$2.00$ \\
$2^{8}$&$1.9982e-04$&$2.00$   &$4.5548e-04$&$2.00$    &$6.0172e-04$&$2.00$ \\
$2^{9}$&$6.3269e-05 $&$2.00$  &$1.4417e-04 $&$2.00$   &$1.9039e-04$&$2.00$ \\
$2^{10}$&$1.9527e-05 $&$2.00$ &$4.4492e-05 $&$2.00$   &$5.8762e-05$&$2.00$    \\
$2^{11}$&$5.9069e-06 $&$-$    &$1.3460e-05 $&$-$      &$1.7776e-05$&$-$     \\ \hline
 $\varepsilon$&\multicolumn{2}{c}{$2^{-4}$}&\multicolumn{2}{c}{$2^{-6}$}&\multicolumn{2}{c}{$2^{-10}$}\\\hline\hline
    $N$ &$E_n$&Ord&$E_n$&Ord&$E_n$&Ord\\\hline
$2^{5}$   &$1.9969e-02$&$2.43$     &$1.9957e-02$&$2.41$ &$1.9957e-02$&$2.41$ \\
$2^{6}$   &$5.7712e-03$&$2.02$     &$5.8271e-03$&$2.02$ &$5.8271e-03$&$2.02$ \\
$2^{7}$   &$1.9427e-03$&$2.00$     &$1.9616e-03$&$2.00$ &$1.9616e-03$&$2.00$\\
$2^{8}$   &$6.4337e-04$&$2.00$     &$6.4051e-04$&$2.00$ &$6.4051e-04$&$2.00$ \\
$2^{9}$   &$2.0072e-04$&$2.00$     &$2.0266e-04$&$2.00$ &$2.0266e-04$&$2.00$ \\
$2^{10}$  &$6.1950e-05$&$2.00$     &$6.2550e-05$&$2.00$ &$6.2550e-05$&$2.00$ \\
$2^{11}$  &$1.8740e-05$&$-$        &$1.8921e-05$&$-$    &$1.8921e-05$&$-$  \\\hline
 $\varepsilon$&\multicolumn{2}{c}{$2^{-12}$}&\multicolumn{2}{c}{$2^{-20}$}&\multicolumn{2}{c}{$2^{-30}$}\\\hline\hline
\end{tabular}
\caption{Errors $E_N$ and convergence rates Ord for approximate solutions.}\label{tabela}
\end{table}
\end{center}

\noindent \textbf{The explanations about the figures.}
\noindent In Figure \eqref{slika1}, \eqref{slika2} and  \eqref{slika3}, the plots  of the exact solution of the problem \eqref{uvod1}--\eqref{uvod3} and the approximate solutions \eqref{mr2}  are presented, for the values of the parameters $N=32$ and $\varepsilon=2^{-4},\,2^{-6},\,2^{-10}$, respectively, while in figures  \eqref{slika4}, \eqref{slika5}  and \eqref{slika6} graphics of exact and numerical solution  \eqref{mr2} were given for the values of the parameters $N=64,\,128,\,256$ and $\varepsilon=2^{-10},$ respectively. In figure \eqref{slika1}, \eqref{slika2} and \eqref{slika3} one can notice an increase of the error value, or differences in the graphs between the exact and numerical solutions, while in Figure \eqref{slika1} it is very difficult to distinguish between the exact and numerical solutions \eqref{mr2}, in Figure  \eqref{slika3} the deviation between the numerical and exact solution can be seen. From the presented graphs it is evident that there is a decrease of the error value due to an increase in the number of points  $N.$  \\\\

\noindent In Figures \eqref{slika13}, \eqref{slika14} and \eqref{slika15a} the plots of the exact \eqref{uvod1}--\eqref{uvod3} and approximate solution \eqref{eksp12} are given. For the calculation of the approximate solutions we used $N=32$ points, while the value of the perturbation parameter was $\varepsilon=2^{-4},\,2^{-6},\,2^{-10},$ respectively. From the presented graphics it can be seen a decrease of perturbation parameter $\varepsilon,$ with a constant value of the number of points $N$ a value of the error is slightly  increasing. However, this increase is smaller than in the case of use of approximate solutions  \eqref{mr2}.  In  the Figure \eqref{slika16}, \eqref{slika17} and \eqref{slika18} there  are graphs of the correct solution of the problems \eqref{uvod1}--\eqref{uvod3} and approximate solutions. Graphs on all three figures are obtained for a fixed value of parameter $\varepsilon,$ while approximate solution is obtained by using $N=64,\,128,\, 256$ number of points, respectively.  \\\\

\noindent In Figures \eqref{slika11a}, \eqref{slika11c} and  \eqref{slika11e} the plots of the error of the approximate solutions  \eqref{mr2} are represented, while in Figures  \eqref{slika11b}, \eqref{slika11d} and  \eqref{slika11f} are graphs of the error of the approximate solution \eqref{eksp12}. Side by side are graphs of the errors of the approximate solution, to the left is \eqref{mr2}, while on the right are approximate solution \eqref{eksp12} for the same values of the parameter $\varepsilon$ and $N$. From the graph we can see that values of the error agree with the theoretical results. In the graph, on the right side is a value of the error from the order $\mathcal{O}\left(N^{-1}\right),$ while on the graphs from the right side is a value of the error from the order $\mathcal{O}\left(N^{-2}\ln^2N  \right),$ and therefore in this way we have a confirmation of the theoretical results.

\section{Conclusion}
\noindent In this paper we performed a construction of approximate solutions for singularly--perturbed boundary value problem \eqref{uvod1}--\eqref{uvod3}. First, we calculated a discrete approximate solution, i.e. the value of approximate solution in points of the mesh, and then we constructed an approximate solution by using a representation of the exact solution via Green's functions. Order of the value of the error is  $\mathcal{O}\left(N^{-1}\right)$ in the maximum norm. The basis functions are exponential.  From Theorem \ref{th1} we can see that the value of errors in this way constructed approximate solution is in the part of the domain where lies boundary layer of order $\mathcal{O}\left(\ln^2N/N^2\right),$ while out of the layer are of order   $\mathcal{O}\left(1/N\right).$ In order to gain the approximate solution with the smallest error, basis function of the exponential type of the outer boundary layer is replaced with linear functions. Error in this case is in the order $\mathcal{O}\left(\ln^2N/N^2\right),$ also in the maximum norm.

\newpage
\begin{figure}[H]
\centering
\begin{subfigure}[b]{.45\textwidth}
 \includegraphics[width=\textwidth]{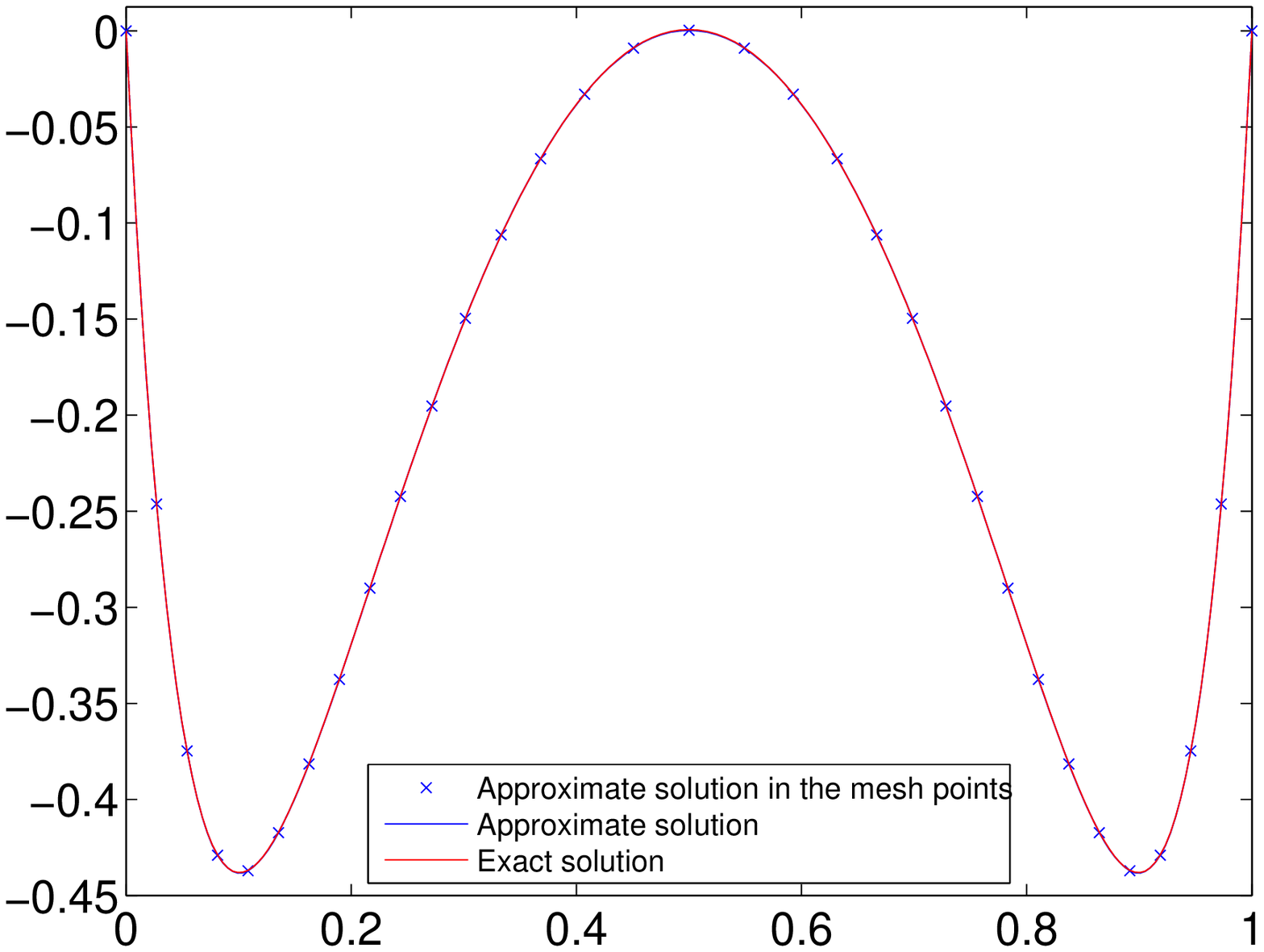}
 \caption{$N=32,\:\varepsilon=2^{-4}$}
\label{slika1}
\end{subfigure}
\begin{subfigure}[b]{.45\textwidth}
 \includegraphics[width=\textwidth]{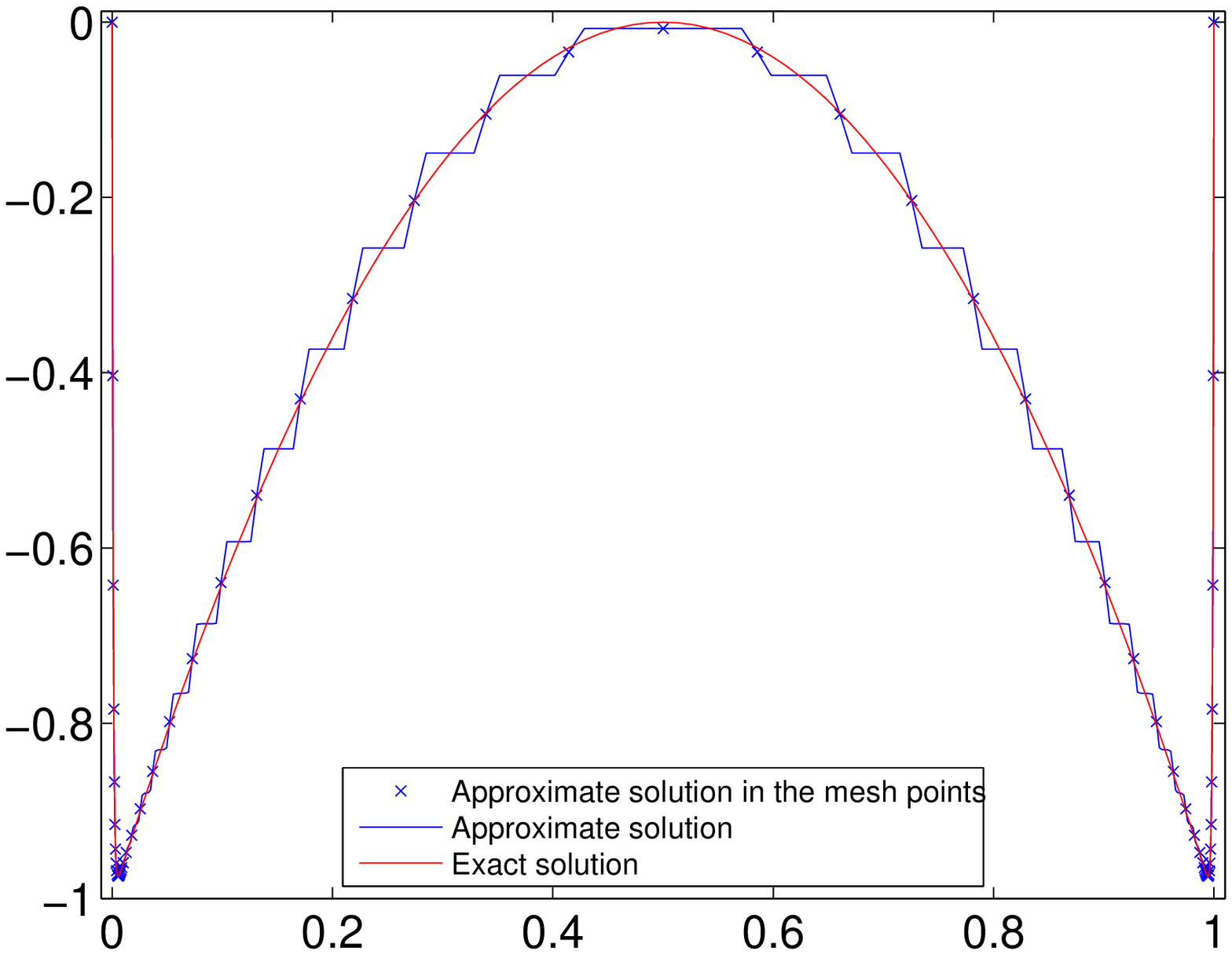}
 \caption{$N=64,\:\varepsilon=2^{-10}$}
\label{slika4}
\end{subfigure}
\label{slika34z}
\centering
\begin{subfigure}[b]{.45\textwidth}
 \includegraphics[width=\textwidth]{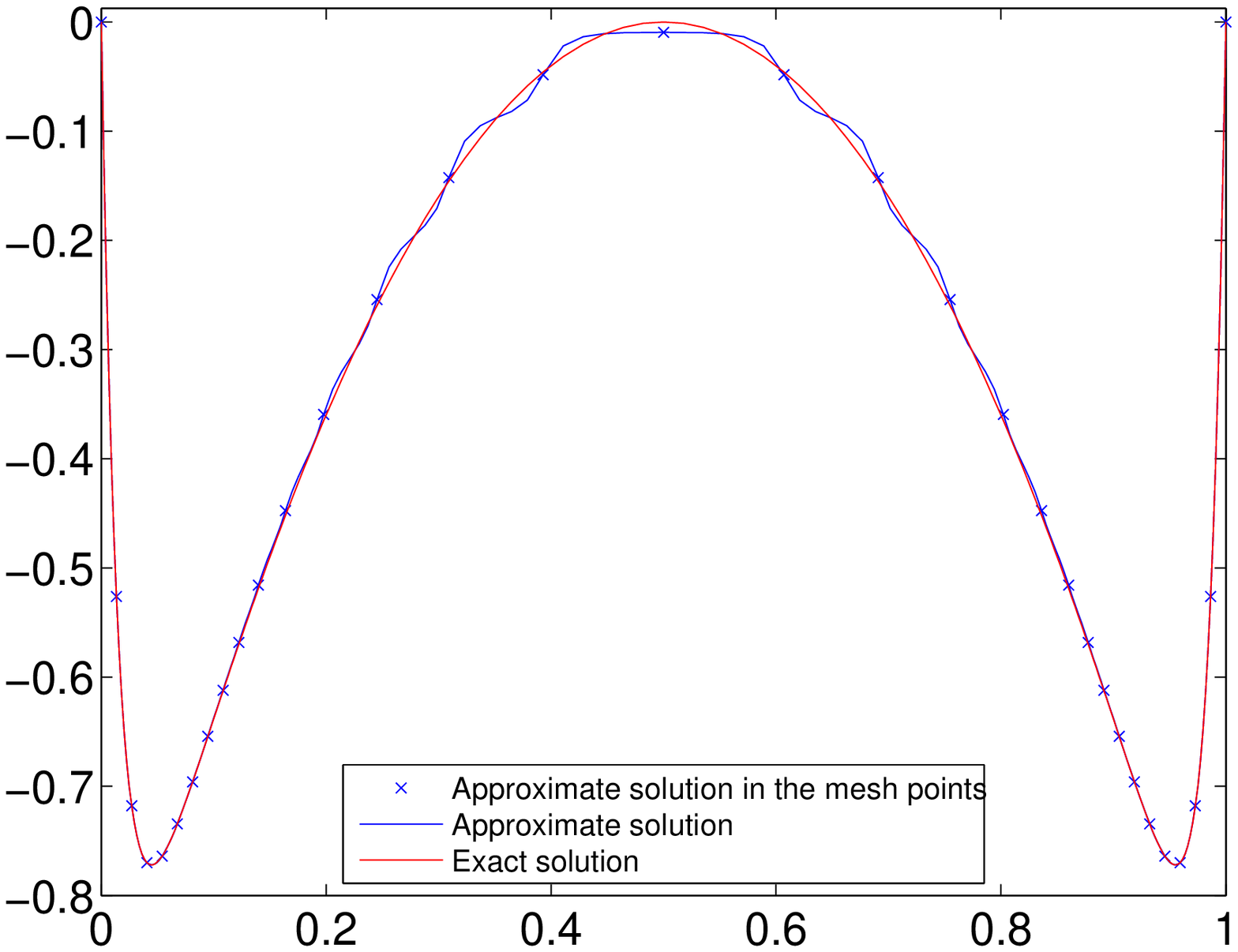}
 \caption{$N=32,\:\varepsilon=2^{-6}$}
\label{slika2}
\end{subfigure}
\label{slika12z}
\centering
\begin{subfigure}[b]{.45\textwidth}
 \includegraphics[width=\textwidth]{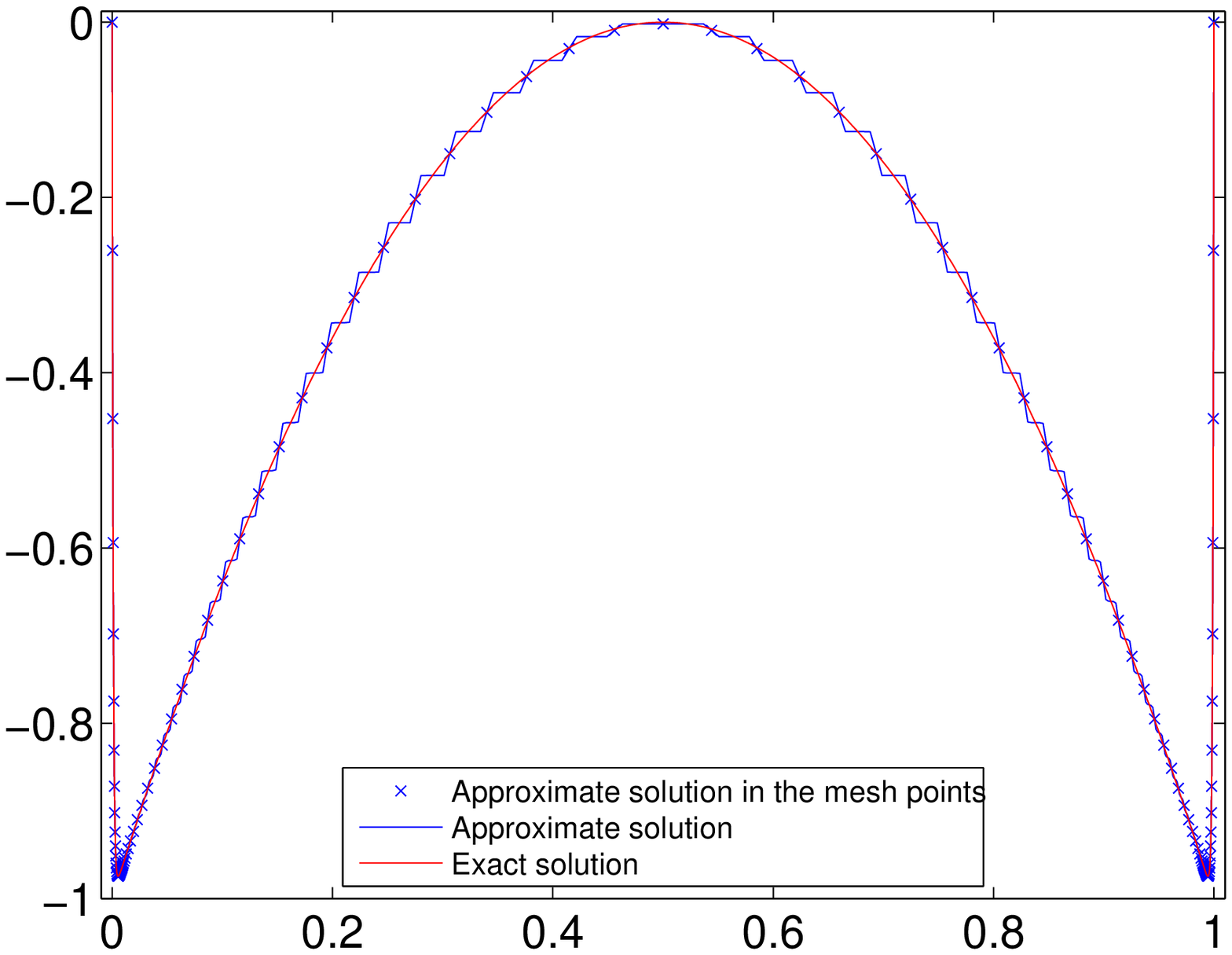}
 \caption{$N=128,\:\varepsilon=2^{-10}$}
\label{slika5}
\end{subfigure}
\begin{subfigure}[b]{.45\textwidth}
 \includegraphics[width=\textwidth]{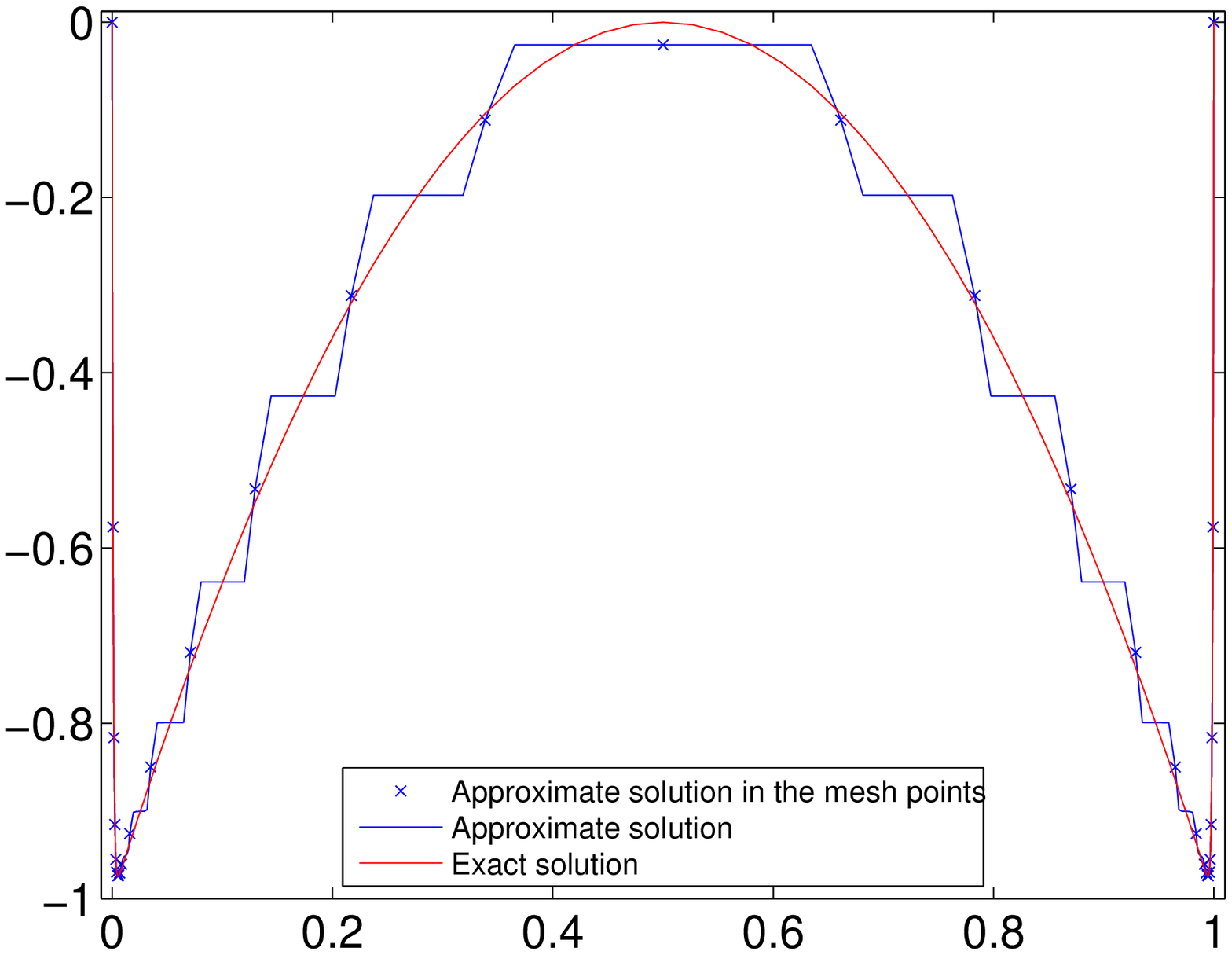}
 \caption{$N=32,\:\varepsilon=2^{-10}$}
\label{slika3}
\end{subfigure}
\begin{subfigure}[b]{.45\textwidth}
 \includegraphics[width=\textwidth]{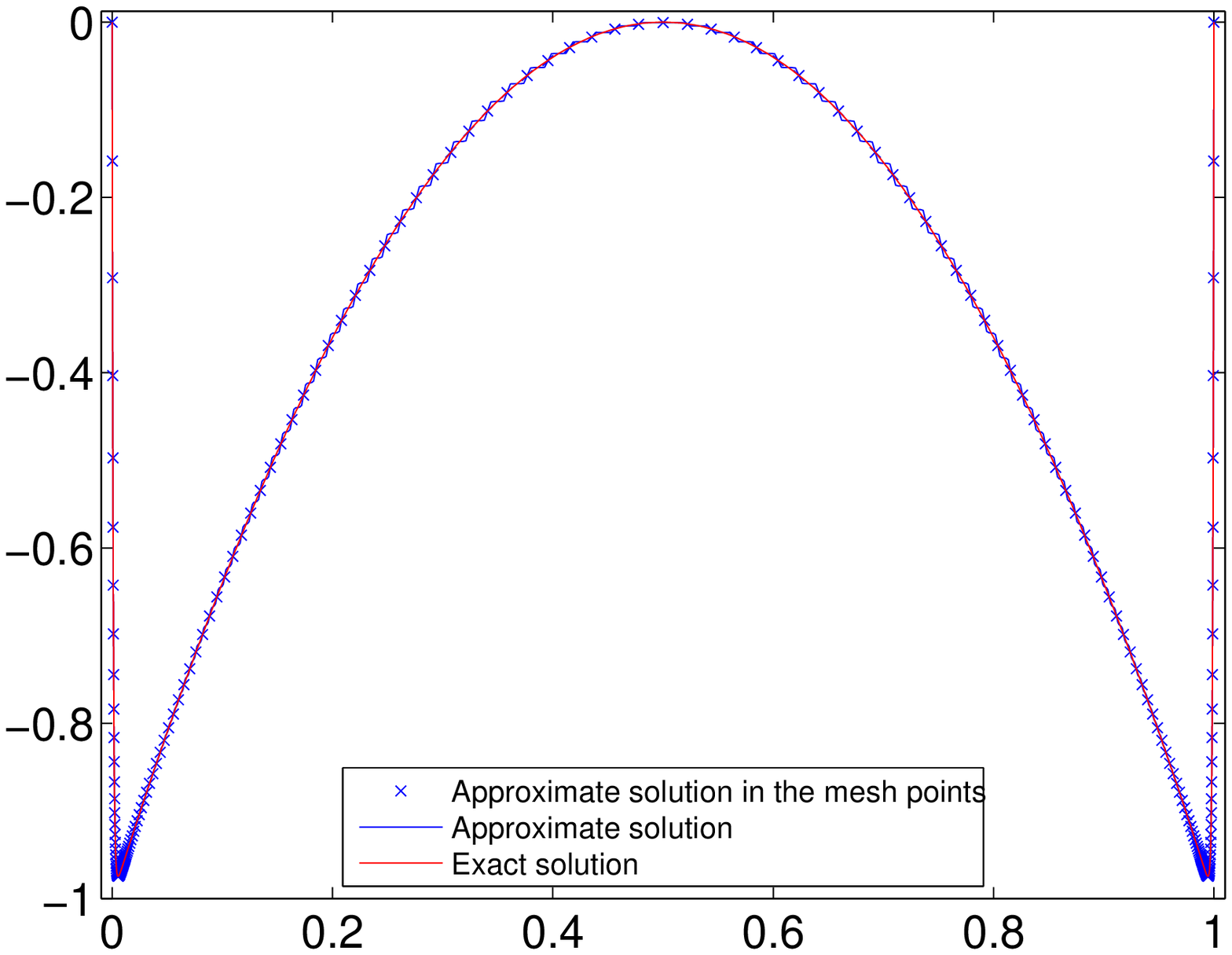}
 \caption{$N=256,\:\varepsilon=2^{-10}$}
\label{slika6}
\end{subfigure}
\caption{ }
\label{slika56z}
\end{figure}

\begin{figure}[H]
\centering
\begin{subfigure}[b]{.45\textwidth}
 \includegraphics[width=\textwidth]{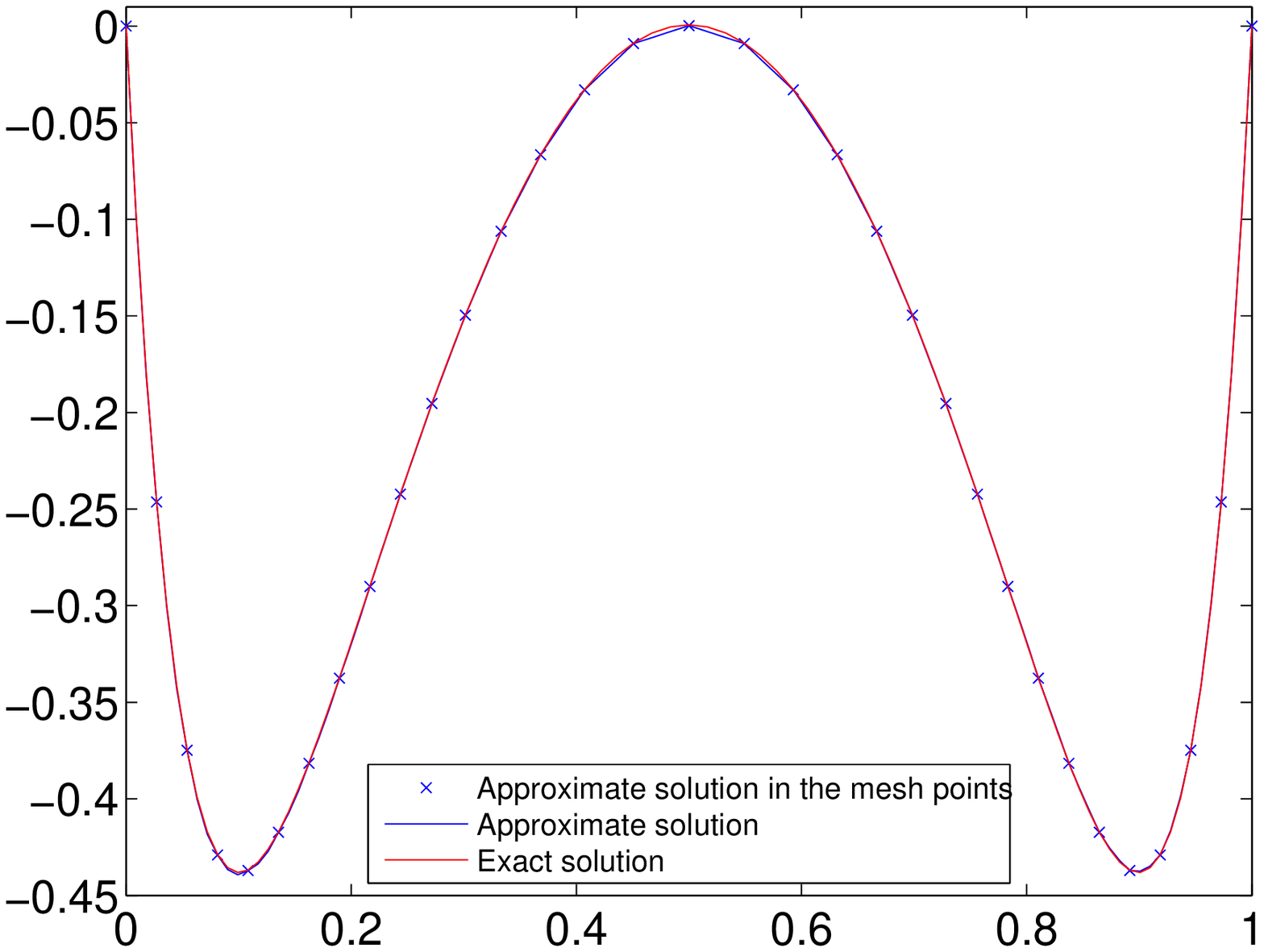}
 \caption{$N=32,\:\varepsilon=2^{-4}$}
\label{slika13}
\end{subfigure}
\begin{subfigure}[b]{.45\textwidth}
 \includegraphics[width=\textwidth]{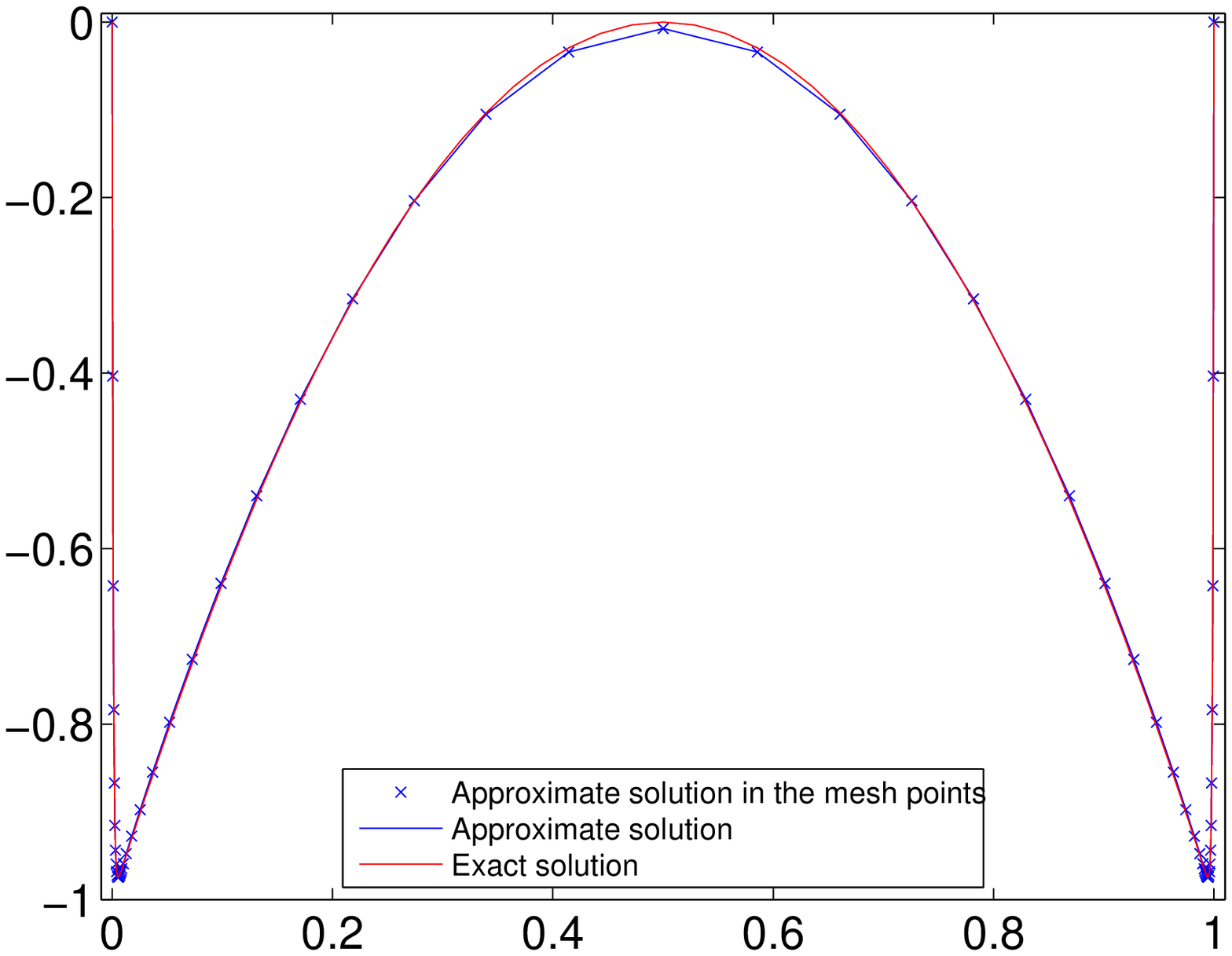}
 \caption{$N=64,\:\varepsilon=2^{-4}$}
\label{slika16}
\end{subfigure}
\label{slika15b16z}
\begin{subfigure}[b]{.45\textwidth}
 \includegraphics[width=\textwidth]{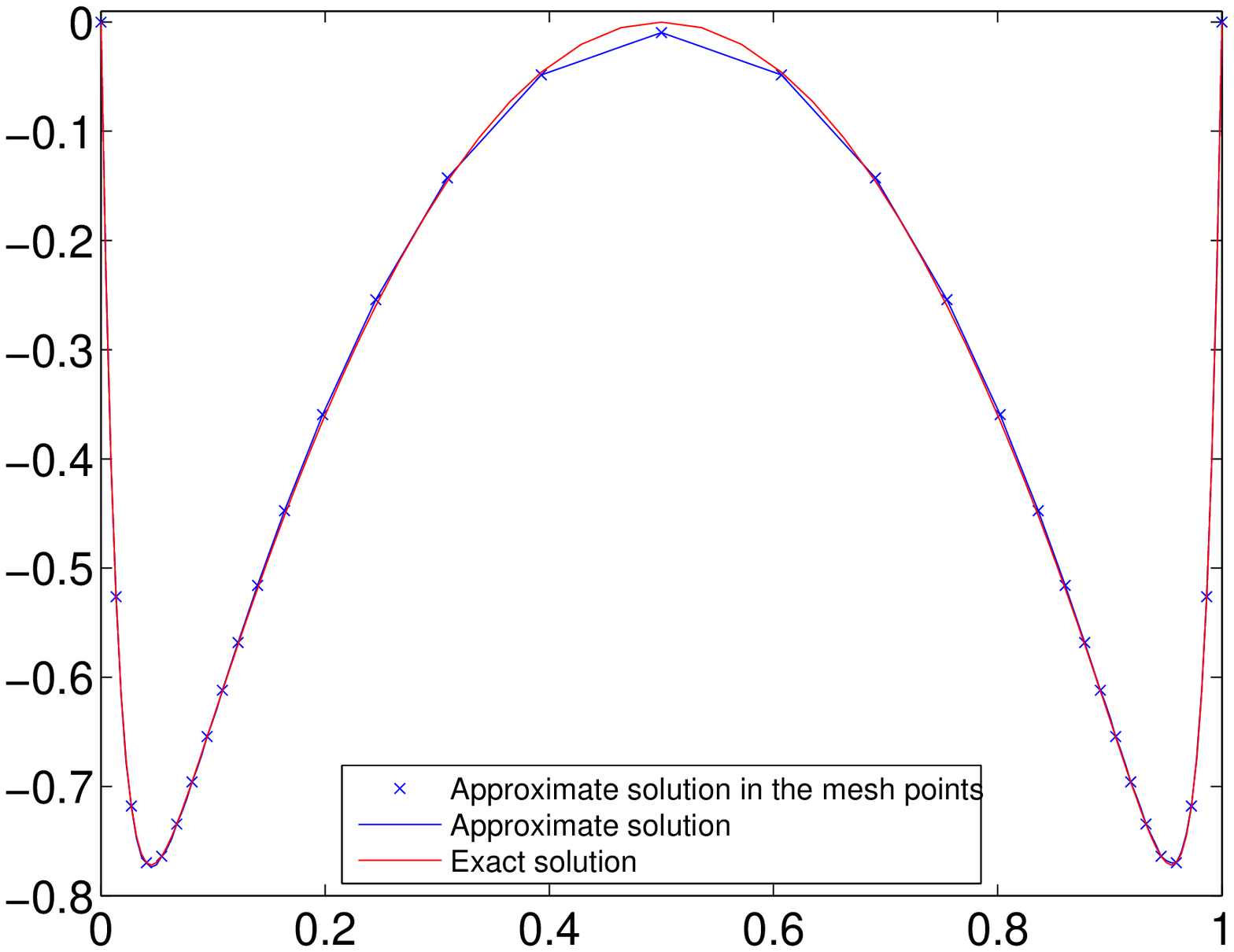}
 \caption{$N=32,\:\varepsilon=2^{-6}$}
\label{slika14}
\end{subfigure}
\label{slika1314z}
\centering
\begin{subfigure}[b]{.45\textwidth}
 \includegraphics[width=\textwidth]{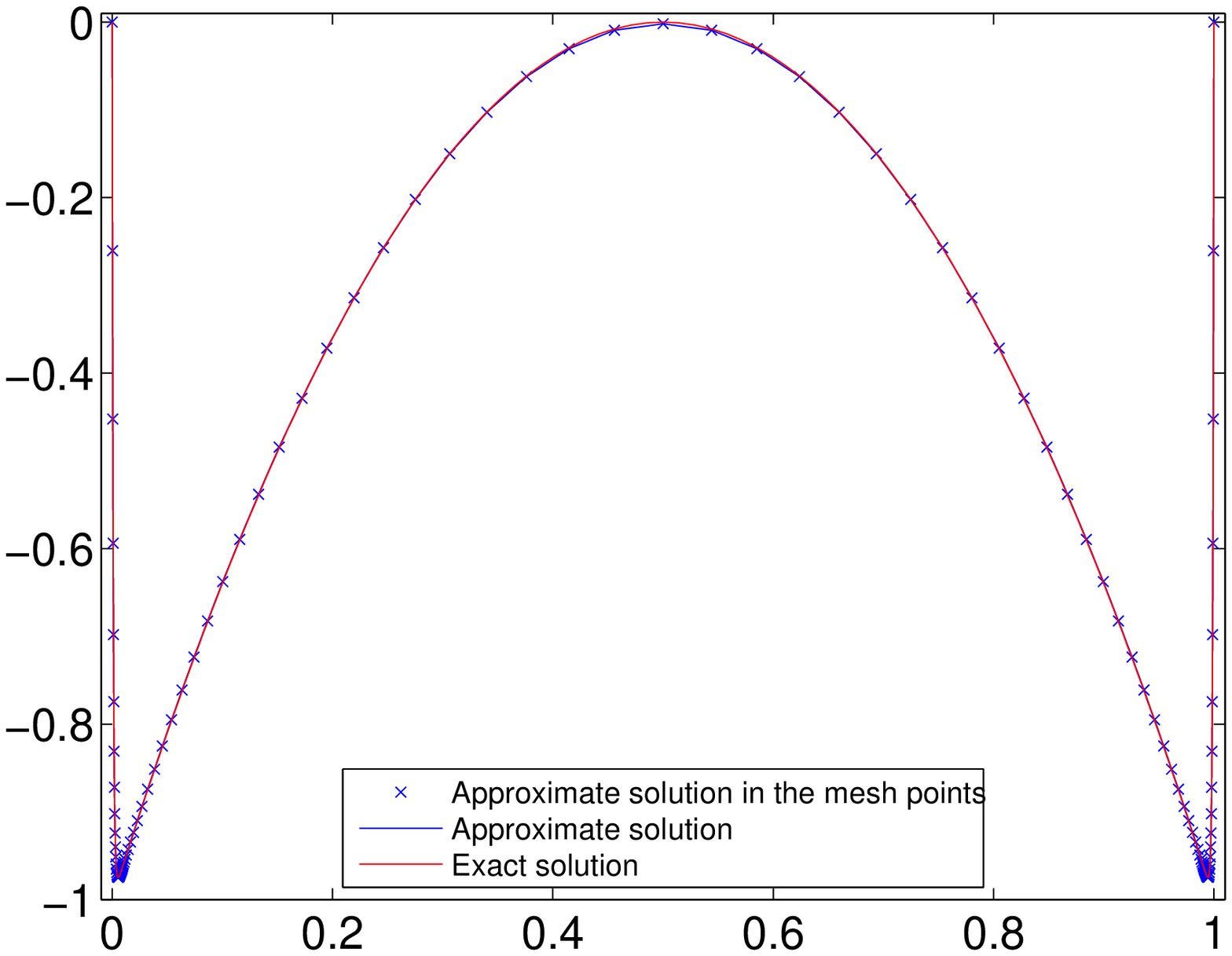}
 \caption{$N=128,\:\varepsilon=2^{-10}$}
\label{slika17}
\end{subfigure}
\centering
\begin{subfigure}[b]{.45\textwidth}
 \includegraphics[width=\textwidth]{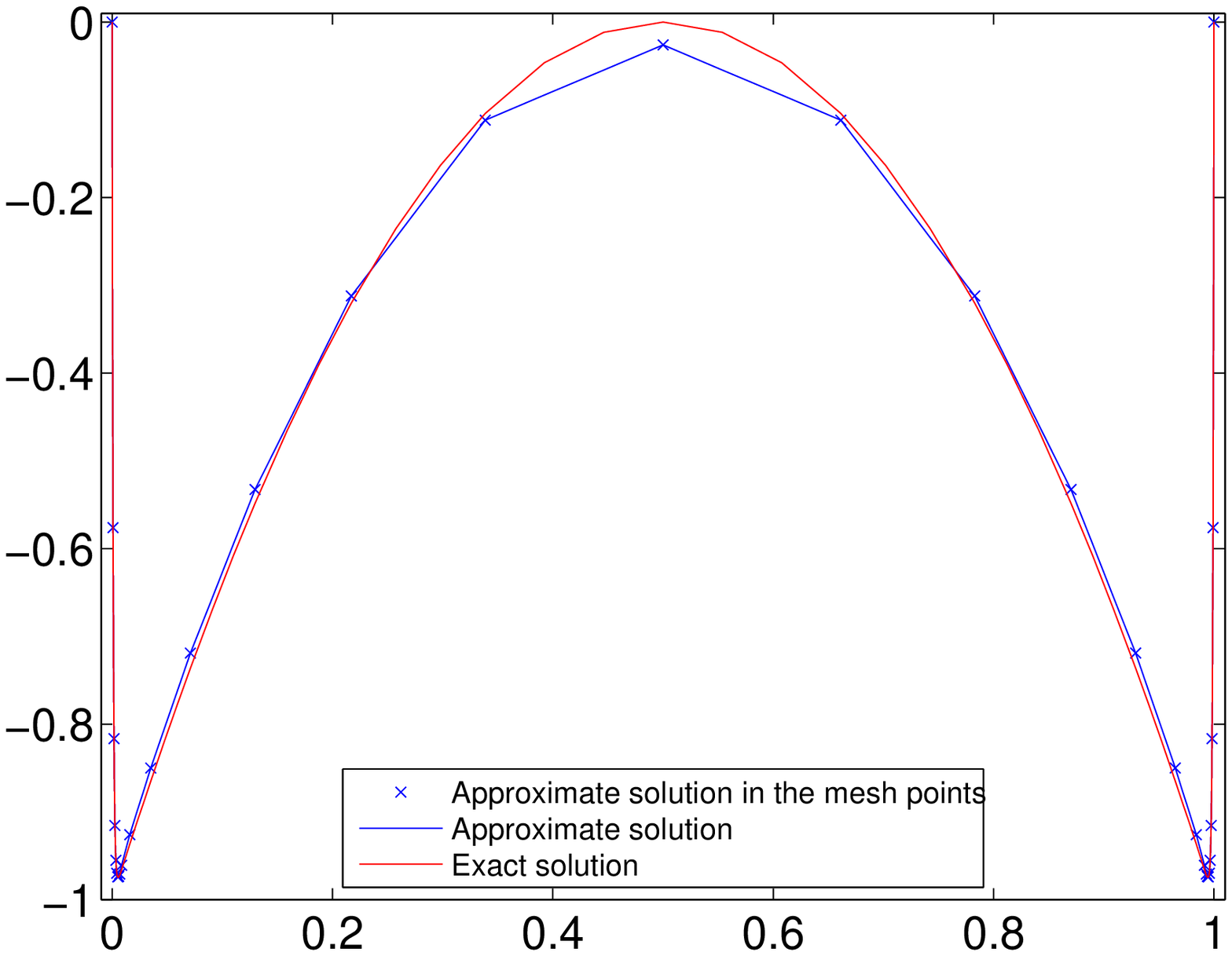}
 \caption{$N=32,\:\varepsilon=2^{-10}$}
\label{slika15a}
\end{subfigure}
\begin{subfigure}[b]{.45\textwidth}
 \includegraphics[width=\textwidth]{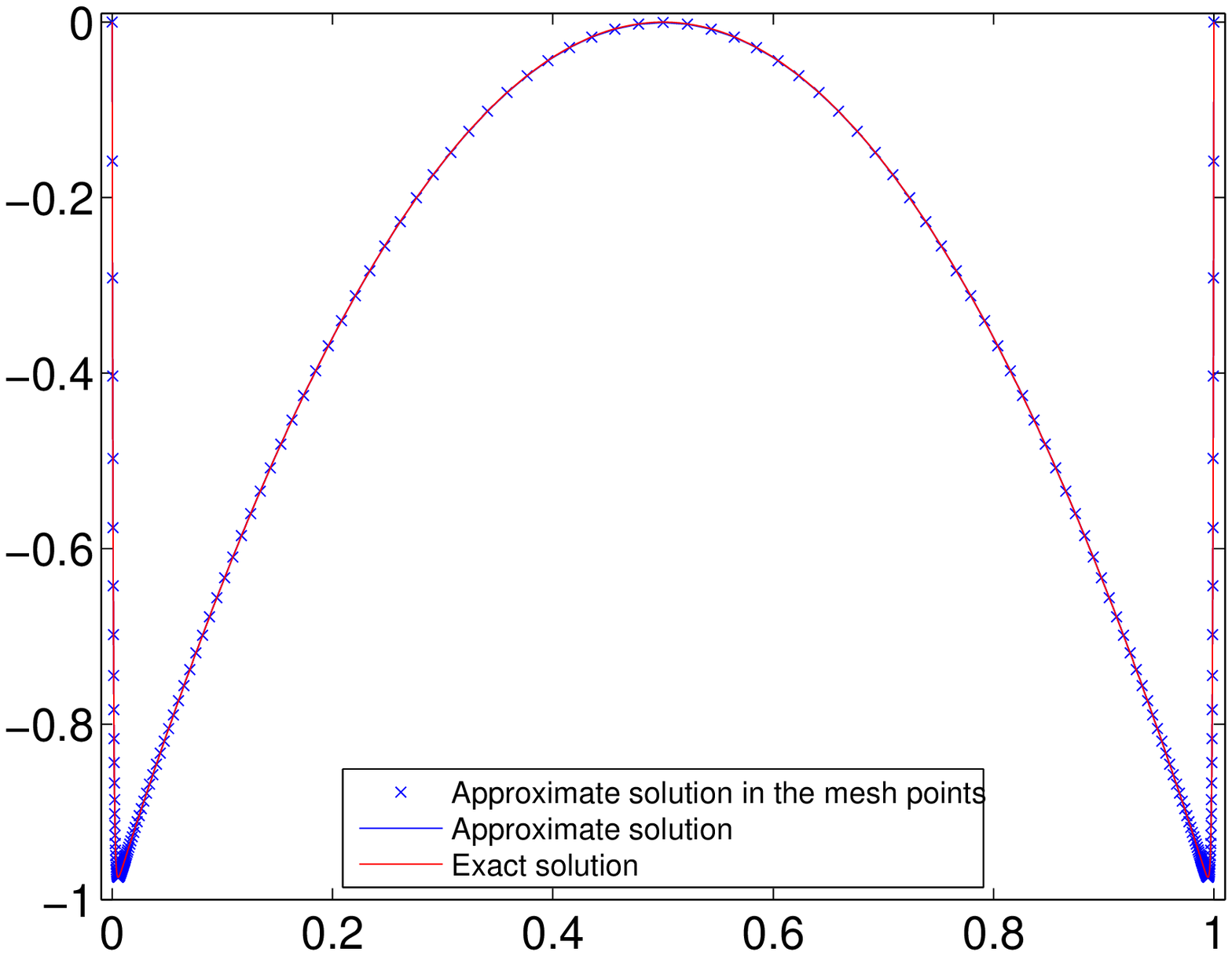}
 \caption{$N=256,\:\varepsilon=2^{-10}$}
\label{slika18}
\end{subfigure}
\caption{ }
\label{slika1718z}
\end{figure}

\begin{figure}[H]
\centering
\begin{subfigure}[b]{.45\textwidth}
 \includegraphics[width=\textwidth]{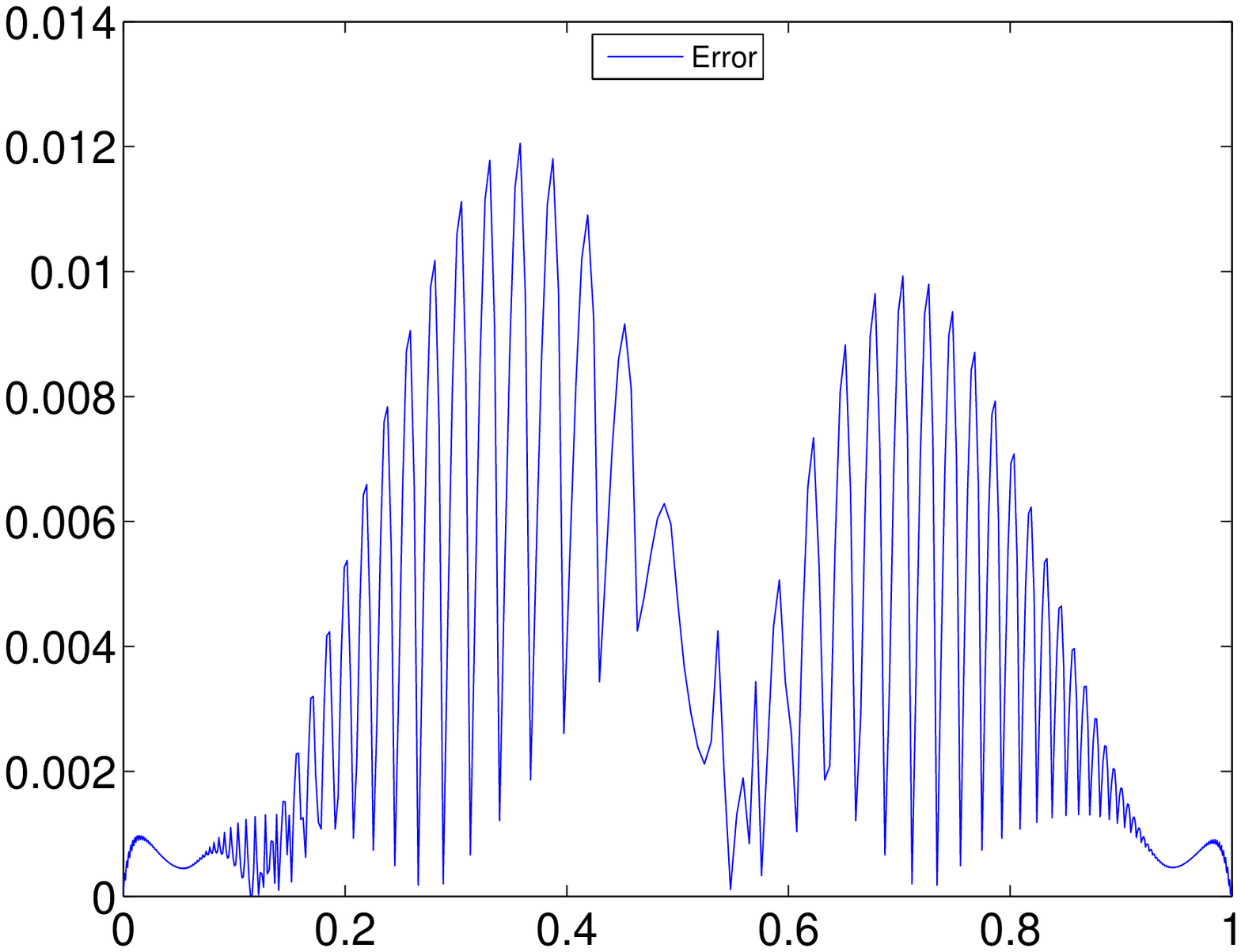}
 \caption{$N=128,\:\varepsilon=2^{-6}$}
\label{slika11a}
\end{subfigure}
\centering
\begin{subfigure}[b]{.45\textwidth}
 \includegraphics[width=\textwidth]{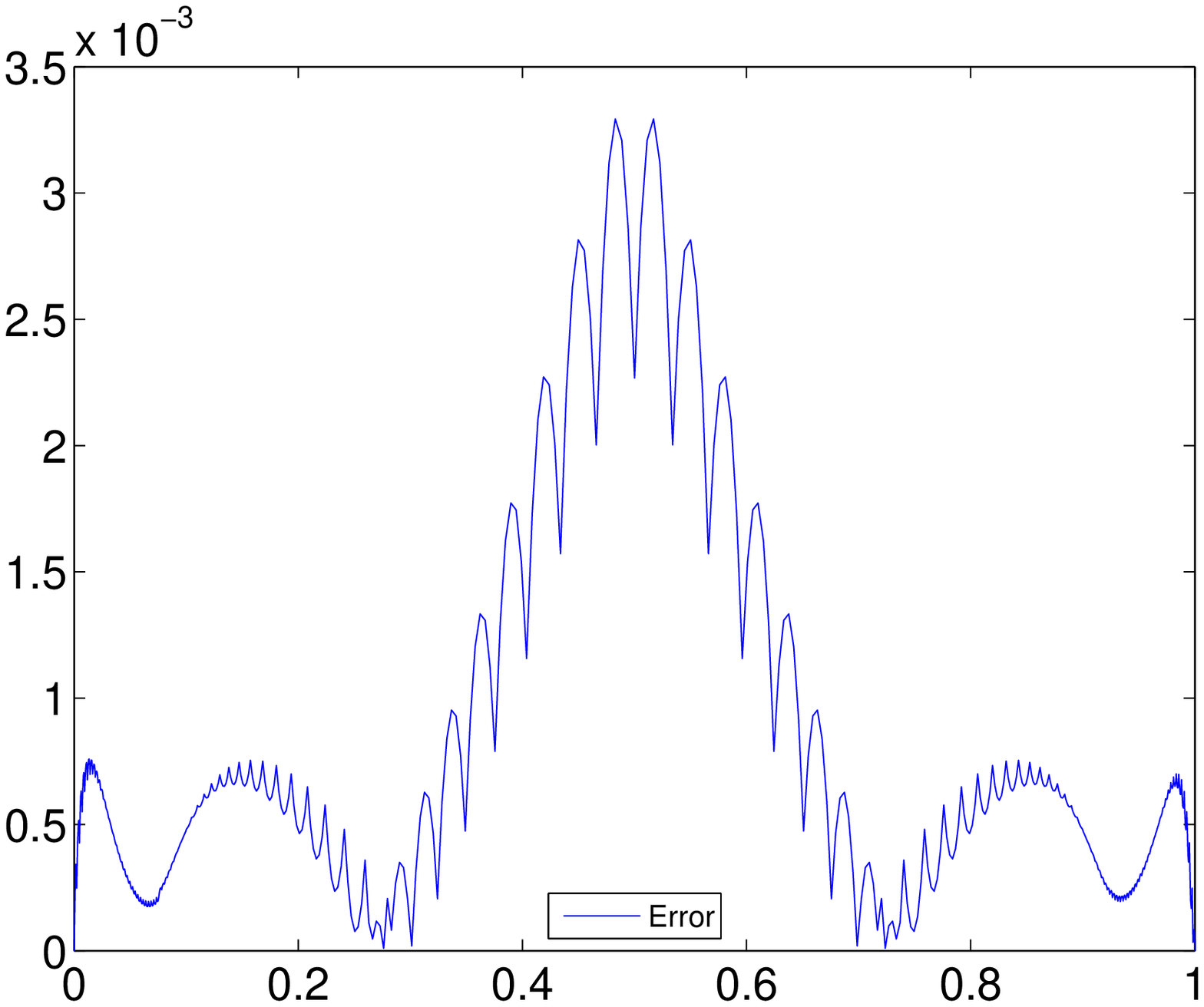}
 \caption{$N=128,\:\varepsilon=2^{-6}$}
\label{slika11b}
\end{subfigure}
\centering
\begin{subfigure}[b]{.45\textwidth}
 \includegraphics[width=\textwidth]{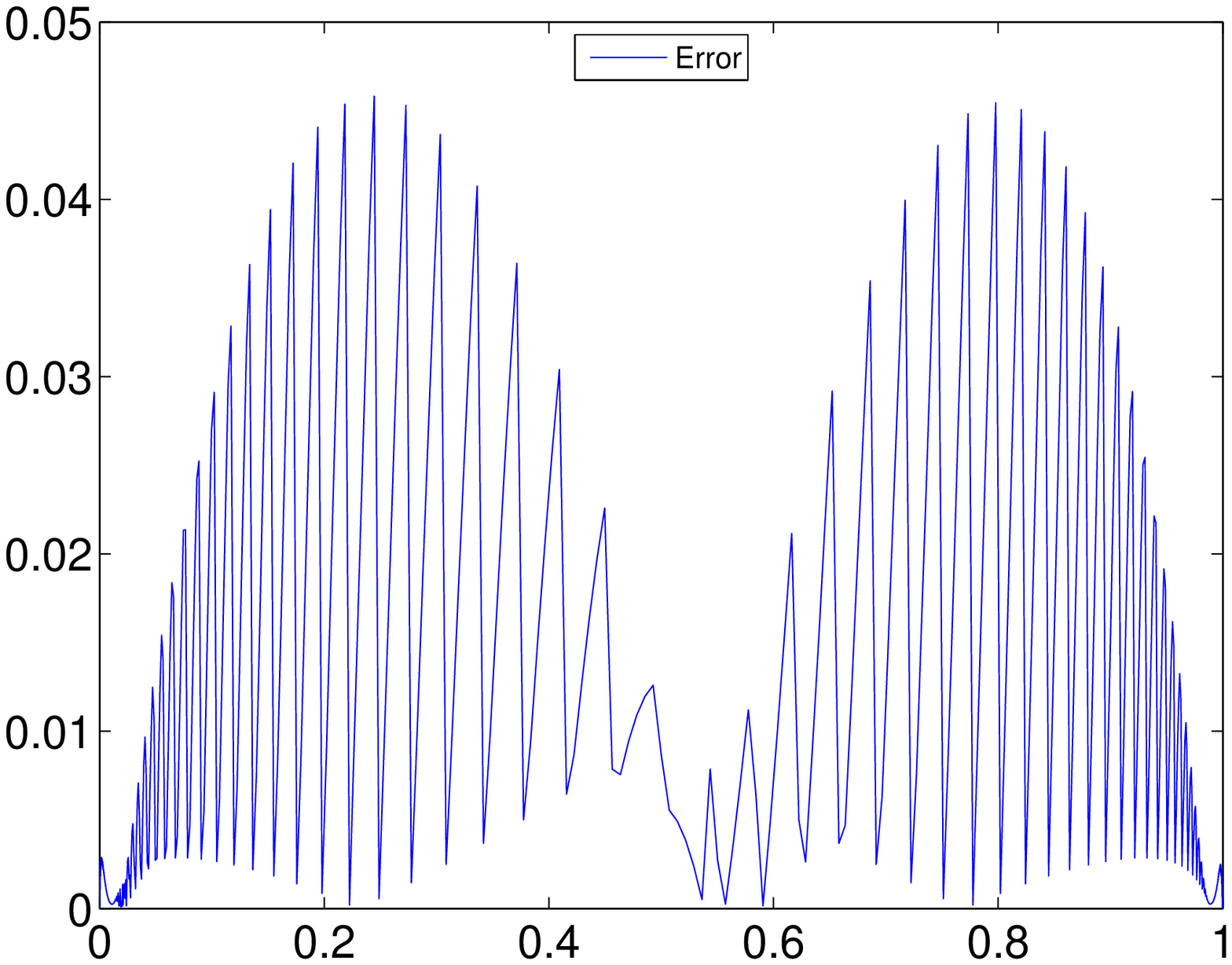}
 \caption{$N=128,\:\varepsilon=2^{-9}$}
\label{slika11c}
\end{subfigure}
\centering
\begin{subfigure}[b]{.45\textwidth}
 \includegraphics[width=\textwidth]{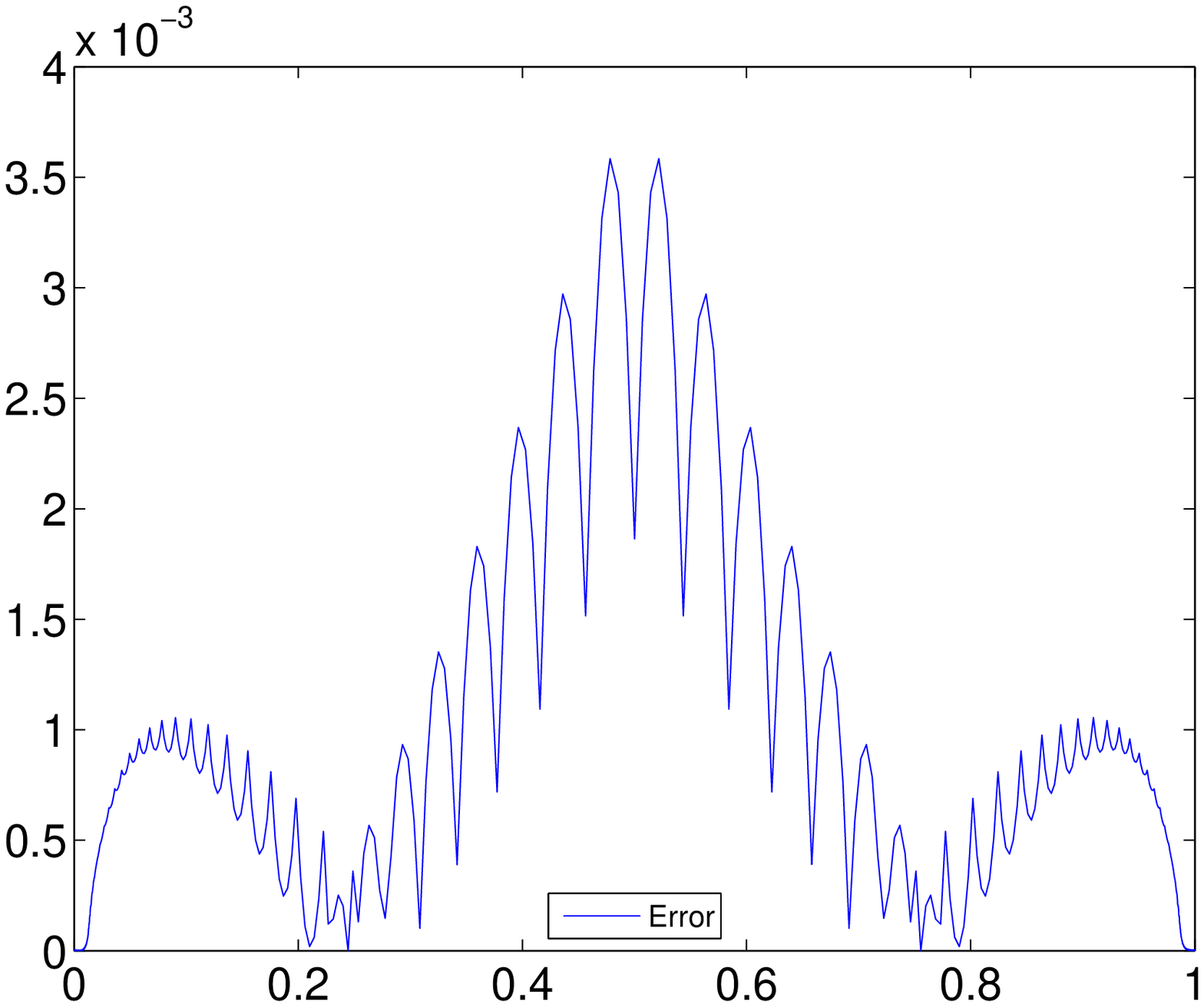}
 \caption{$N=128,\:\varepsilon=2^{-9}$}
\label{slika11d}
\end{subfigure}
\centering
\begin{subfigure}[b]{.45\textwidth}
 \includegraphics[width=\textwidth]{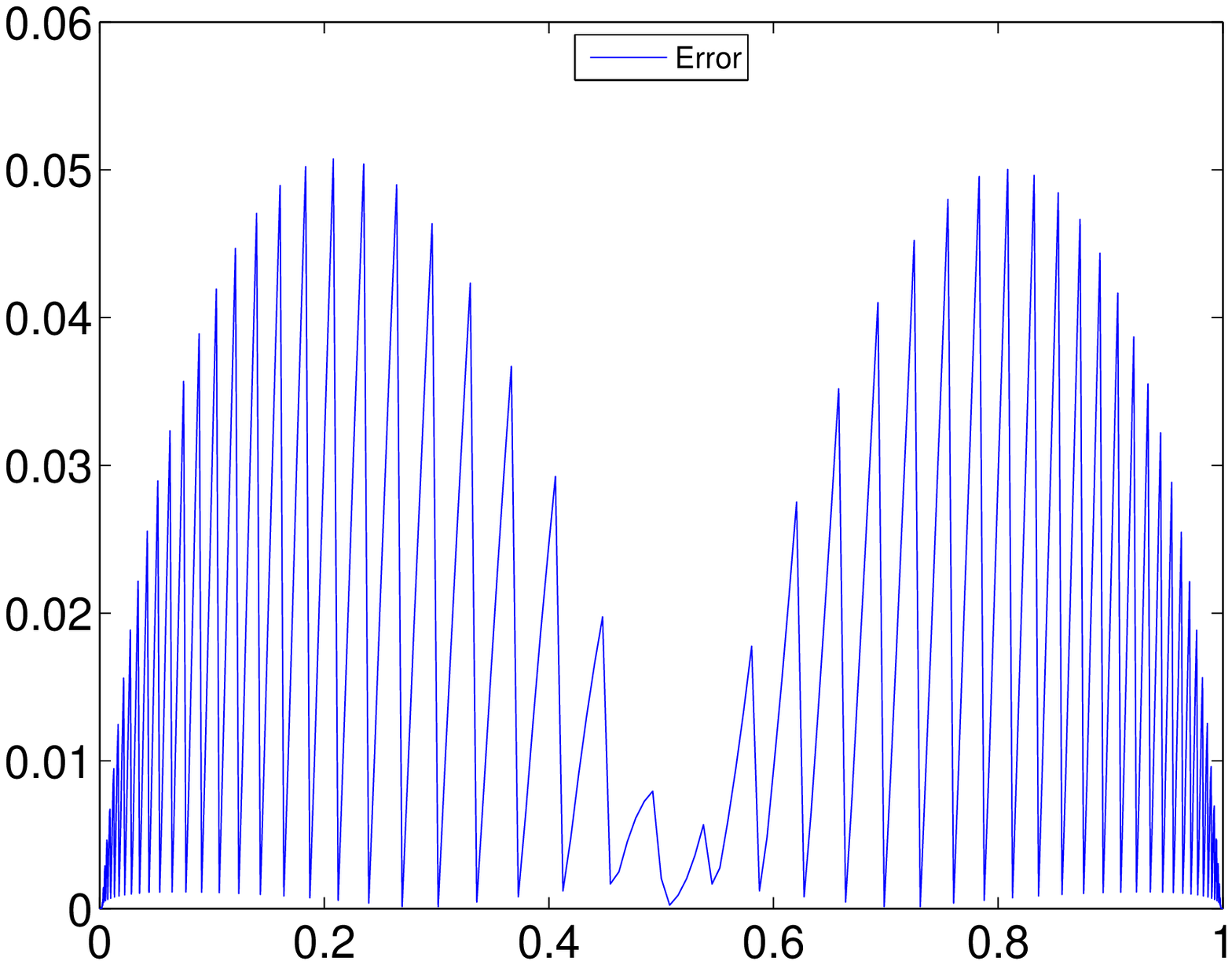}
 \caption{$N=128,\:\varepsilon=2^{-12}$}
\label{slika11e}
\end{subfigure}
\centering
\begin{subfigure}[b]{.45\textwidth}
 \includegraphics[width=\textwidth]{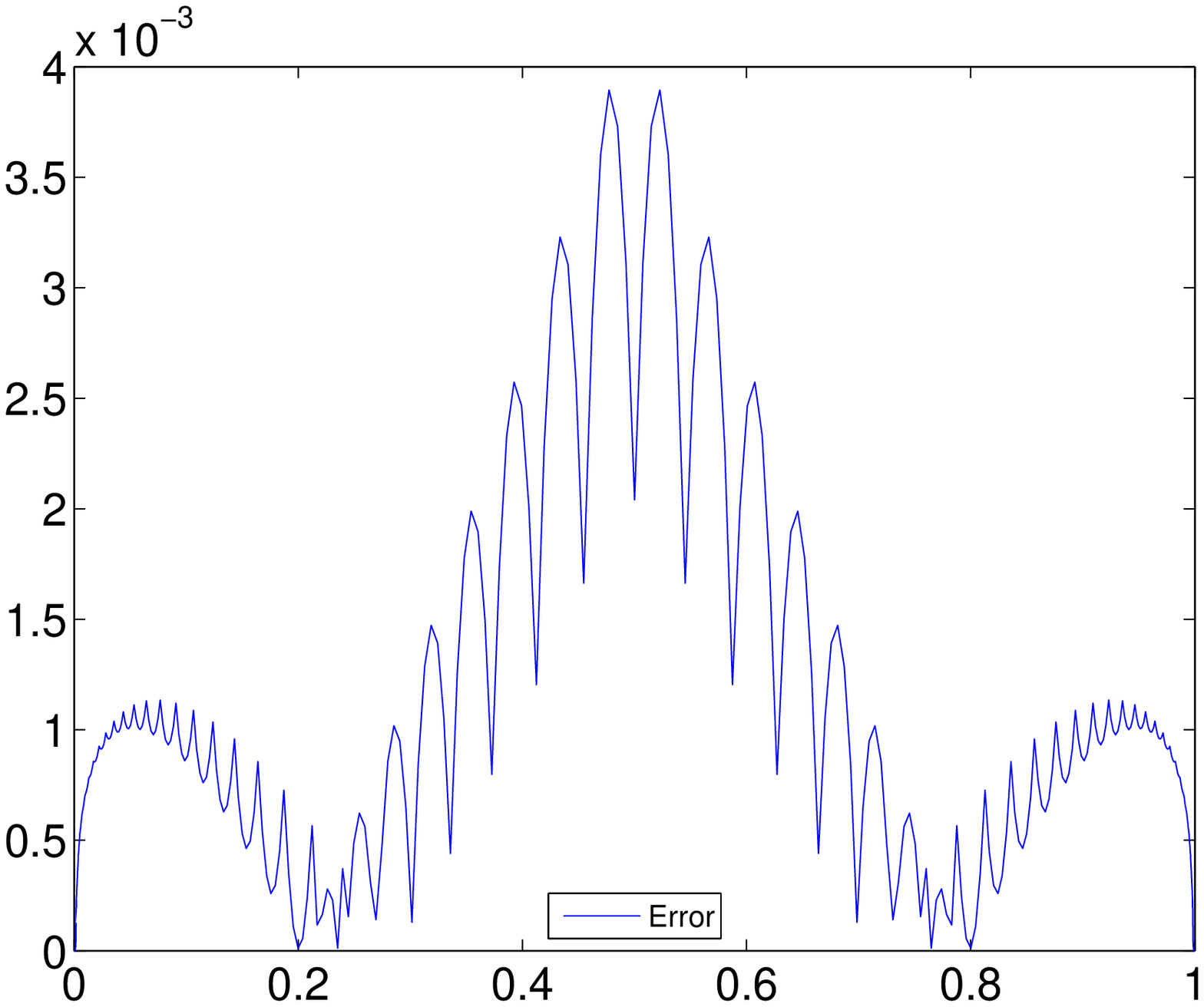}
 \caption{$N=128,\:\varepsilon=2^{-12}$}
\label{slika11f}
\end{subfigure}
\caption{The graphs error}
\label{slikagreska}
\end{figure}

\bibliographystyle{amsplain}
\bibliography{Manuscript}
\end{document}